\documentclass[a4paper]{article}
\usepackage{amsthm,amsmath,amssymb}
\usepackage{makeidx}
\usepackage{pifont}
\usepackage{color}
\usepackage{textcomp}
\usepackage{graphicx}
\usepackage{pgf,tikz}

\newtheorem{teor}{Theorem}[section]
\newtheorem{defin}[teor]{Definition}
\newtheorem{lemm}[teor]{Lemma}
\newtheorem{osse}[teor]{Remark}
\newtheorem{prop}[teor]{Proposition}
\newtheorem{defi}[teor]{Definition}
\newtheorem{coro}[teor]{Corollary}
\newtheorem{prob}[teor]{Problem}

\newcommand{\bele}{\begin{lemm}\begin{sl}}
\newcommand{\enle}{\end{sl}\end{lemm}}
\newcommand{\bedef}{\begin{defi}\begin{sl}}
\newcommand{\eddef}{\end{sl}\end{defi}}
\newcommand{\bete}{\begin{teor}\begin{sl}}
\newcommand{\ente}{\end{sl}\end{teor}}
\newcommand{\beos}{\begin{osse}\begin{rm}}
\newcommand{\eddos}{\end{rm}\end{osse}}
\newcommand{\bepr}{\begin{prop}\begin{sl}}
\newcommand{\empr}{\end{sl}\end{prop}}
\newcommand{\bepro}{\begin{prob}\begin{rm}}
\newcommand{\empro}{\end{rm}\end{prob}}
\newcommand{\bede}{\begin{defin}\begin{sl}}
\newcommand{\edde}{\end{sl}\end{defin}}
\newcommand{\beco}{\begin{coro}\begin{sl}}
\newcommand{\enco}{\end{sl}\end{coro}}

\newcommand{\quext}{\quad\text}

\newcommand{\RR}{\mathbb{R}}
\newcommand{\NN}{\mathbb{N}}

\newcommand{\beeq}[1]{\begin{equation}\label{#1}}
\newcommand{\eddeq}{\end{equation}}

\newcommand{\beeqa}[1]{\begin{eqnarray}\label{#1}}
\newcommand{\eddeqa}{\end{eqnarray}}

\newcommand{\beal}[1]{\begin{align}\label{#1}}
\newcommand{\eddal}{\end{align}}

\newcommand{\bespl}[1]{\begin{split}\label{#1}}
\newcommand{\edspl}{\end{split}}

\newcommand{\bega}[1]{\begin{gather}\label{#1}}
\newcommand{\edga}{\end{gather}}

\newcommand{\beeqax}{\begin{eqnarray*}}
\newcommand{\eddeqax}{\end{eqnarray*}}

\def\qed{\ifmmode 
  \else \leavevmode\unskip\penalty9999 \hbox{}\nobreak\hfill
  \fi
  \quad\hbox{\hskip.5em\vrule width.4em height.6em depth.05em\hskip.1em}}
\def\endproofsym{\qed}
\renewenvironment{proof}[1][Proof]{\trivlist\item[\hskip\labelsep{\hskip0pt
    {\normalfont\scshape#1.}\hskip .321429\parindent}]\ignorespaces}
{\endproofsym\endtrivlist}
\def\endnobox{\def\endproofsym{}\end{proof}\def\endproofsym{\qed}}

\newcommand{\no}{\nonumber}

\newcommand{\beeqao}{\begin{eqnarray}\no}
\newcommand{\bealo}{\begin{align}\no}
\newcommand{\besplo}{\begin{split}\no}
\newcommand{\begao}{\begin{gather}\no}

\newcommand{\eps}{\varepsilon}

\newcommand{\duav}[1]{\langle{#1}\rangle}

\newcommand{\bm}{\boldsymbol{m}}

\newcommand{\dt}{\partial_t}

\newcommand{\perogni}{\forall\,}
\newcommand{\esiste}{\exists\,}

\newcommand{\io}{\int_\Omega}

\newcommand{\iTT}{\int_0^T}

\newcommand{\iTo}{\iTT\!\io}

\newcommand{\OO}{_{\Omega}}

\newcommand{\bzero}{\boldsymbol{0}}

\newcommand{\fhi}{\varphi}

\newcommand{\lhs}{left hand side}
\newcommand{\rhs}{right hand side}

\DeclareMathOperator{\per}{per}
\DeclareMathOperator{\dive}{div}
\DeclareMathOperator{\deriv}{d}

\DeclareMathOperator{\dist}{dist}
\DeclareMathOperator{\Dist}{Dist}

\newcommand{\LDHD}{L^2(0,T;H^2(\Omega))}

\let\TeXchi\chi
\def\chi{{\setbox0 \hbox{\mathsurround0pt
$\TeXchi$}\hbox{\raise\dp0 \copy0 }}}

\newcommand{\zzn}{_{0,n}}

\newcommand{\teta}{\vartheta}

\newcommand{\calH}{{\mathcal H}}

\newcommand{\calA}{{\mathcal A}}

\newcommand{\calE}{{\mathcal E}}

\newcommand{\calN}{{\mathcal N}}

\newcommand{\calB}{{\mathcal B}}

\newcommand{\calS}{{\mathcal S}}

\newcommand{\calV}{{\mathcal V}}

\newcommand{\calW}{{\mathcal W}}

\newcommand{\ZZZ}{{\mathcal Z}}

\newcommand{\tetagiu}{\underline{\teta}}

\newcommand{\dit}{\deriv\!t}
\newcommand{\dis}{\deriv\!s}

\newcommand{\dir}{\deriv\!r}

\newcommand{\ddt}{\frac{\deriv\!{}}{\dit}}
\newcommand{\dds}{\frac{\deriv\!{}}{\dis}}

\newcommand{\bu}{\boldsymbol{u}}

\newcommand{\vt}{\vartheta}

\newcommand{\vu}{\boldsymbol{u}}
\newcommand{\vm}{\boldsymbol{m}}

\newcommand{\barvt}{\overline{\vt}}

\newcommand{\barfhi}{\overline{\fhi}}
\newcommand{\barmu}{\overline{\mu}}
\newcommand{\barz}{\overline{z}}
\newcommand{\barvu}{\overline{\vu}}

\def\div{{\rm div}}

\def\th{{\vartheta}}
\def\tho{{{\vartheta}_\Omega}}

\def \l {\langle}
\def \r {\rangle}

\newenvironment{bettirev}{\color{blue}}{\color{black}}
\newcommand{\bber}{\begin{bettirev}}
\newcommand{\eber}{\end{bettirev}}

\newenvironment{girev}{\color{red}}{\color{black}}
\newcommand{\III}{\begin{girev}}
\newcommand{\EEE}{\end{girev}}

\newenvironment{michelarev}{\color{red}}{\color{black}}
\newcommand{\bmicr}{\begin{michelarev}}
\newcommand{\emicr}{\end{michelarev}}

\numberwithin{equation}{section}
\begin{document}

\title{Regularity and long-time behavior for a thermodynamically consistent model for
 complex fluids in two space dimensions\thanks{All the authors are partially supported by the GNAMPA
 (Gruppo Nazionale per l'Analisi Matematica, la Probabilit\`a e le loro Applicazioni)
 of INdAM (Istituto Nazionale di Alta Matematica), through the project GNAMPA 2016 ``Regolarit\`a e comportamento asintotico di soluzioni di equazioni
 paraboliche'' (coord.\ Prof.\ S.~Polidoro) and by the University of Modena and Reggio Emilia through the project FAR2015 ``Equazioni differenziali:
 problemi evolutivi, variazionali ed applicazioni'' (coord. Prof. S.~Polidoro)}}

\author{
{Michela Eleuteri \thanks{University of Modena and Reggio Emilia, Dipartimento di Scienze Fisiche, Informatiche e Matematiche, via Campi 213/b, I-41125 Modena (Italy).
E-mail: \textit{michela.eleuteri@unimore.it}}}
\and
{Stefania Gatti\thanks{University of Modena and Reggio Emilia, Dipartimento di Scienze Fisiche, Informatiche e Matematiche, via Campi 213/b, I-41125 Modena (Italy).
E-mail: \textit{stefania.gatti@unimore.it}}}
\and
{Giulio Schimperna\thanks{Dipartimento di Matematica ``F. Casorati'',
Universit\`a degli Studi di Pavia,
via Ferrata 5, I-27100 Pavia, Italy.
E-mail: \textit{giusch04@unipv.it}}}}

\maketitle

\begin{abstract}\noindent
 We consider a thermodynamically consistent model for the evolution of
 thermally conducting two-phase incompressible fluids. Complementing
 previous results, we prove additional regularity properties of solutions
 in the case when the evolution takes place in the two-dimensional flat torus
 with periodic boundary conditions. Thanks to
 improved regularity, we can also prove uniqueness
 and characterize the long-time behavior of trajectories showing existence
 of the global attractor in a suitable phase-space.
\end{abstract}

\smallskip


\smallskip



\section{Introduction}
\label{sec:intro}

We consider here a mathematical model for two-phase flows
of non-isothermal incompressible fluids in a bounded
container $\Omega\subset\RR^2$. The model consists in a PDE
system describing the evolution of the unknown variables $\vu$ (macroscopic
velocity), $\fhi$ (order parameter), $\mu$ (chemical potential),
$\vt$ (absolute temperature) and taking the form
\begin{align}\label{incom}
  & \dive \vu = 0, \\
 \label{ns}
  & \vu_t + \vu \cdot \nabla \vu + \nabla p
    = \Delta \vu - \dive ( \nabla \fhi \otimes \nabla \fhi ),\\
 \label{CH1}
  & \fhi_t + \vu \cdot \nabla \fhi = \Delta \mu, \\
 \label{CH2}
  & \mu = - \Delta \fhi + F'(\fhi) - \vt, \\
 \label{calore}
  & \vt_t + \vu \cdot \nabla \vt + \vt \big( \fhi_t + \vu \cdot \nabla \fhi \big)
    - \dive(\kappa(\vt)\nabla \vt) = | \nabla \vu |^2 + | \nabla \mu |^2.
\end{align}
Relation~\eqref{ns}, with the incompressibility constraint \eqref{incom},
represents a variant of the Navier-Stokes equations;
\eqref{CH1}-\eqref{CH2} correspond
to a form of the Cahn-Hilliard system \cite{CH} for phase separation,
while \eqref{calore} is the internal energy equation describing
the evolution of temperature. As usual, the variable $p$
in the Navier-Stokes system~\eqref{ns} represents the (unknown) pressure.
The function $F$ whose derivative appears in \eqref{CH2} is a possibly non-convex
potential whose minima represent the least energy configurations of
the phase variable. Here we will assume that $F$ is smooth
and has a power-like growth at infinity.
Finally, the function $\kappa(\vt)$ in~\eqref{calore} denotes the
heat conductivity coefficient, assumed to grow
at infinity like a sufficiently high power of $\vt$
(see \eqref{hp:kappa} below).

The system is highly nonlinear and
contains complicated coupling terms; however these features
arise naturally and are directly related to the thermodynamical
consistency of the model. In particular, the quadratic terms
on the \rhs\ of \eqref{calore}, which constitute the main
difficulty in the mathematical analysis, describe the
heat production coming from dissipation of kinetic and chemical
energy, respectively. We may also note that transport effects are
admitted for all variables in view of the occurrence of
material derivatives in \eqref{ns}, \eqref{CH1}
and~\eqref{calore}.  In order to avoid complications
related to interactions with the boundary, we will assume here
$\Omega=[0,1]\times[0,1]$ to be the two-dimensional flat torus.
Correspondingly, we will take periodic boundary conditions
for all unknowns.

System~\eqref{incom}-\eqref{calore} has been first introduced in 
\cite{ERS} and can be considered as a coupling
between the Navier-Stokes equations and the thermodynamically consistent
model for phase transitions proposed by M.~Fr\'emond in \cite{BFL00}
and extensively studied in recent years (see, for instance,
\cite{CLS02, LaSS02, LSS01, LSS02, SLX} and the references therein,
we also quote \cite{EKK15} for applications to elastoplasticity with hysteresis
and \cite{FFRS} for liquid crystals). Other nonisothermal models for
phase-changing fluids can be obtained by linearization around the
critical temperature, which simplifies the mathematical analysis but
gives rise at least to a partial loss of thermodynamical consistency.

A mathematical study of \eqref{incom}-\eqref{calore}
has been first attempted in \cite{ERS,ERS2}, which refer
to the three- and two-dimensional setting, respectively.
Referring to these articles for a physical justification
of the equations and a more comprehensive
survey of the related mathematical literature, here
we just recall that, in \cite{ERS}, existence of solutions was shown
(under slightly different assumptions on coefficients with respect to
those considered here) for a {\sl very weak}\/ formulation, where the
heat equation \eqref{calore} was replaced by a differential equality
accounting for the balance of total energy, complemented with
an ``entropy production'' differential inequality. This approach follows an idea
originally devised in \cite{BFM} for heat conducting fluids
and later used in other contexts (see, e.g., \cite{FFRS,FRSZ} and \cite{RR}
for applications to nematic liquid crystals and damaging models,
respectively). Indeed, in view of the upper regularity threshold for the
3D Navier-Stokes system, the quadratic terms on the \rhs\ of \eqref{calore}
(and particularly the one depending on $\nabla \vu$) may be only estimated in
$L^1$, which gives rise to defect measures when taking the limit
in an approximation scheme. Actually, the method of \cite{BFM}
permits to overcome this difficulty, but
at the price of dealing with a weaker concept of solution.

Later, the two-dimensional case was analyzed in \cite{ERS2}
where stronger results were obtained. Actually, in 2D
one can deduce additional estimates and, in particular,
the quadratic terms in \eqref{calore} can be controlled in
$L^2$ (though the procedure to get such a bound is not trivial).
Using that information the authors of \cite{ERS2} could obtain
existence of ``strong solutions'' satisfying the equations
\eqref{incom}-\eqref{calore} with the initial and boundary conditions
in the usual (distributional) sense (hence avoiding the occurrence
of differential inequalities). On the other hand,
the results of \cite{ERS2} leave several open questions
which we would like to answer here. In particular,
taking essentially the same assumptions on coefficients and data
considered there (here we only need to specify in a more precise
way the hypotheses on the initial temperature),
we will extend and complement the results of \cite{ERS2}
in the following three directions:
\begin{itemize}
 \item We will improve the results on regularity, defining a class of slightly
 smoother solutions (called {\sl ``stable solutions''} in the sequel);
 \item We will prove that for stable solutions uniqueness holds (hence
 we have well-posedness in this class);
 \item We will characterize the long-time behavior of stable solutions
 showing that they constitute a strongly continuous dynamical process which
 admits the global attractor.
\end{itemize}
\noindent
We refer the reader to the next Section~\ref{sec:motiv} for a
rigorous motivation and a detailed explanation of the aspects under
which our result improve and complement those given in \cite{ERS2}.
In particular, the terminology ``stable'' solutions will be clarified
there. Here, we conclude the introduction with the plan of the rest of the paper:
Section~\ref{sec:main} is devoted to presenting
the precise statements of our results, whose proofs are split
into several parts. Namely, in Section~\ref{sec:rego} further regularity
on finite time intervals is discussed; in
Section~\ref{sec:uniq} uniqueness is proved for ``stable solutions'';
finally, in Section~\ref{quattro} the long-time behavior of trajectories
is characterized and the existence of non-empty $\omega$-limit sets
is shown. Moreover, the global attractor in the sense of infinite-dimensional
dynamical systems is proved to exist.


\section{Assumptions and motivation}
\label{sec:motiv}

We start by introducing some notation.
Recalling that $\Omega=[0,1]\times[0,1]$,
we denote as $H:=L^2_{\per}(\Omega)$ the space of
functions in $L^2(\RR^2)$ which are $\Omega$-periodic
(i.e., $1$-periodic both in $x_1$ and in $x_2$). Analogously,
we set $V:=H^1_{\per}(\Omega)$. The spaces $H$ and $V$
are endowed with the norms of $L^2(\Omega)$
and $H^1(\Omega)$, respectively. For brevity, the norm
in $H$ will be simply indicated by $\| \cdot \|$.
We will note by $\| \cdot \|_X$ the norm
in the generic Banach space $X$. The symbol
$\duav{\cdot,\cdot}$ will indicate the duality
between $V'$ and $V$ and $(\cdot,\cdot)$
will stand for the scalar product of~$H$.
We also write $L^p(\Omega)$ in place of $L^p_{\per}(\Omega)$,
and the same for other spaces; indeed, no confusion should
arise since periodic boundary conditions are assumed
to hold for all unknowns. We denote by $H^m_{\rm per}(\Omega)$ (or for
brevity simply $H^m(\Omega)$) the space of functions which are $H^m_{\rm loc}(\mathbb{R}^2)$
and $\Omega$-periodic. For $m \in \mathbb{N}$ they are introduced by means of the corresponding
Fourier series and then they can be extended for general $m \in \mathbb{R}$, $m \ge 0$.
In particular, for $m = 0$ we have $H^0_{\rm per}(\Omega) = L^2_{\rm per}(\Omega)$.

For any function $v\in H$, we will set
\begin{equation}\label{voo}
  v\OO:= \frac1{|\Omega|} \io v = \io v,
\end{equation}
to indicate the spatial mean of $v$. If the integral is replaced with
a duality, the above can be extended to $v\in V'$.
The symbols $V_0$, $H_0$ and $V_0'$ denote the subspaces
of $V$, $H$ and, respectively, $V'$ containing the function(al)s
having zero spatial mean. Then, the distributional operator $(-\Delta)$ is
invertible if seen as a mapping from $V_0$ to $V_0'$ and its
inverse will be indicated by $\calN$.

Still for brevity, we use the
same notation for indicating vector-valued (or tensor-valued)
function spaces and related norms. For instance,
writing $\vu \in H$, we will in fact mean
$\vu\in L^2_{\per}(\Omega)^2$. Also the incompressibility
constraint \eqref{incom} will not be emphasized in the
notation for functional spaces, unless on occurrence: in that case we set
\[
  \mathbb{V} := \{\mathbf{v} \in V_0:
   \dive \mathbf{v} = 0\}, \quad \qquad \mathbb{H} := \{\mathbf{v} \in H_0:~\dive \mathbf{v} = 0\}.
\]
Otherwise, for instance the notation $\vu \in H$
will also implicitly subsume that $\dive \vu = 0$ in the
sense of distributions.
These simplifications will
allow us to shorten a bit some formulas.

Moreover, in the following we will frequently use the following
2D interpolation inequalities:
\begin{align}\label{dis:L4}
  & \| v \|_{L^4(\Omega)}
   \le c \| v \|_{V}^{1/2} \| v \|^{1/2},\\
 \label{dis:Linfty}
  & \| v \|_{L^\infty(\Omega)}
   \le c \| v \|_{H^2(\Omega)}^{1/2} \| v \|^{1/2},\\
	\label{dis:Brezis}
  & \| v \|_{L^r(\Omega)}
   \le c \| v \|^{\frac{2}{r}} \| v \|_{V}^{1 - \frac{2}{r}}, \qquad \qquad \! r \in [1, \infty), \\
	\label{dis:BrezziGilardi}
  & \| v \|_{H^s(\Omega)}
   \le c \| v \|_{H^{s_1}(\Omega)}^{1 - \theta} \| v \|_{H^{s_2}(\Omega)}^{\theta},
	\qquad \displaystyle \theta = \frac{s - s_1}{s_2 - s_1},\\
\label{Friedr}
& \|v-v_\Omega\|\leq \, c \|\nabla v\|,
\end{align}
holding for any sufficiently smooth function $v$ and for
suitable embedding constants, all denoted by
the same symbol $c>0$ for brevity.

We will also use the following nonlinear version of the Poincar\'e
inequality (see \cite{GMRS})
\begin{equation}\label{poinc}
  \| v^{p/2} \|_V^2
   \le c_p \big( \| v \|_{L^1(\Omega)}^p
    + \| \nabla v^{p/2} \|^2 \big),
\end{equation}
holding for all nonnegative $v\in L^1(\Omega)$ such that
$\nabla v^{p/2}\in L^2(\Omega)$, and for all $p\in [2, \infty)$.
We also recall that
\begin{equation}\label{inter-gen}
  \| v \| \le c \| \nabla v \|^{1/2} \| v \|_{V'}^{1/2}
   \quext{for all }\, v\in V_0,
\end{equation}
as one can prove simply combining the standard interpolation
inequality $\|v\| \le c \| v \|_V^{1/2} \| v \|_{V'}^{1/2}$
with the Poincar\'e-Wirtinger inequality \eqref{Friedr}.

With the above notation at disposal, we can present
our main assumptions on the nonlinear terms. Basically,
these assumptions will be retained for all our results. They also essentially
coincide with the hypotheses considered in \cite{ERS2} (differences
will be observed on occurrence). First of all, we ask the
configuration potential~$F$ to satisfy:
\begin{align}\label{hp:F1}\tag{F1}
  & F\in C^3(\RR;\RR), \quad
   \liminf_{|r|\to \infty} \frac{F(r)}{|r|} > 0,\\
 \label{hp:F2}\tag{F2}
   & F''(r) \ge - \lambda
    \quext{for some }\,\lambda \ge 0
    \quext{and all }\, r\in \RR,\\
 \label{hp:F4}\tag{F3}
   & | F'''(r) | \le \tilde{c}_F \big( 1 + |r|^{p_F - 1} \big)
    \quext{for some }\, \tilde{c}_F \ge 0,~p_F \ge 1,
    \quext{and all }\, r\in \RR.
\end{align}
We remark that \eqref{hp:F4} implies
\begin{equation} \label{hp:F3}
 | F''(r) | \le c_F \big( 1 + |r|^{p_F} \big)
    \quext{for some }\,c_F \ge 0,~p_F \ge 0,
    \quext{and all }\, r\in \RR.
\end{equation}
Assumption~\eqref{hp:F1} postulates regularity and coercivity
of $F$, \eqref{hp:F2} is $\lambda$-convexity, and \eqref{hp:F4}
prescribes a polynomial growth at infinity. Note that \eqref{hp:F1}
implies that
\begin{equation}
\label{czero}
  F(s) \ge - c_0 \qquad \forall s \in \mathbb{R}
\end{equation}
and some constant $c_0 > 0$. Observe also that in \cite{ERS2}
it was just assumed that $F\in C^2$ (in place of $C^3$) in \eqref{hp:F1};
moreover \eqref{hp:F3} was taken in place of \eqref{hp:F4}.
Here we are asking more regularity because we will look
for smoother solutions.

Next, we assume the heat conductivity to be given
(exactly as in \cite{ERS2}) by
\begin{equation}\label{hp:kappa}\tag{K1}
  \kappa(r) = 1 + r^q, \quad
    q \in [2,\infty), \quad
    r \ge 0.
\end{equation}
Correspondingly, we define
\begin{equation}\label{defiK}
  K(r) := \int_0^r \kappa(s)\,\dis
     = r + \frac{1}{q+1} r^{q+1}, \quad
     r \ge 0.
\end{equation}
We then observe that, for some $k_q > 0$,
\begin{equation}\label{K11}
  \io \kappa(\vt)^2 | \nabla \vt |^2
   = \| \nabla K(\vt) \|^2
   \ge \| \nabla \vt \|^2 + k_q \| \nabla \vt^{q+1} \|^2.
\end{equation}
Let us now discuss the conditions on initial data.
Here we need to give some more words of explanation in view of the
fact that different assumptions on data automatically
generate different concepts of solutions.

To start, we observe that, in view of the periodic boundary conditions and
of the absence of external forces, some physical quantities are necessarily
{\sl conserved}\/ during the evolution. Namely,
any (reasonably defined) solution to the system must satisfy
\begin{equation} \label{cons_medie_energia}
   \vu(t)\OO = \vu(0)\OO, \quad \fhi(t)\OO = \fhi(0)\OO,
    \qquad \mathcal E(\vu(t),\fhi(t),\teta(t))=\mathcal E(\vu(0),\fhi(0),\teta(0)).
\end{equation}
This corresponds to conservation of momentum, of mass,
and of the ``total energy'' $\calE$ defined as
$$
  \mathcal E(\vu,\fhi,\teta)=\io \left( \dfrac12 |\vu|^2+\dfrac12 |\nabla\fhi|^2+F(\fhi)+\teta\right).
$$
Physically speaking, $\calE$ is the sum of the kinetic, interfacial, configuration
and thermal energies. Then, while the first principle of Thermodynamics
yields conservation of $\calE$, the second principle prescribes the production
of entropy, and this fact can be checked directly from the equations simply
by testing \eqref{calore} by $-\teta^{-1}$.
Hence, in order for the initial entropy to be finite, one needs to assume $\teta_0>0$
almost everywhere and $\log\teta_0\in L^1(\Omega)$. This property, together with
the finiteness of the initial energy, leads naturally to define the
``energy-entropy space'' of data as
\begin{equation} \label{spazioen}
   \mathcal H=\big\{ z=(\vu,\fhi,\teta)\in H\times V\times L^1(\Omega):~\dive \vu = 0,~
   \teta>0~\text{a.e.~in $\Omega$},~\log\teta\in L^1(\Omega)\big\}.
\end{equation}
The space $\calH$ is not a Banach space (in view of the occurrence of the
nonlinear logarithm function); nevertheless, it can be endowed with a (complete) metric. Namely, for
$z_i=(\vu_i, \fhi_i, \teta_i) \in \calH$, $i=1,2$, we may set
\begin{equation} \label{distH}
  \dist_{\calH} (z_1,z_2)
    := \|\vu_1-\vu_2\|
     + \|\fhi_1-\fhi_2\|_{V} + \|\teta_1-\teta_2\|_{L^1(\Omega)}
     + \|\log \teta_1 - \log\teta_2 \|_{L^1(\Omega)}.
\end{equation}
Then, what we get from Thermodynamics is that any eventual solution being in $\calH$
at the initial time will remain in $\calH$ in the evolution. Indeed, this regularity
setting corresponds to that of the ``weak solutions'' considered in \cite{ERS}
in the 3D case (in particular, the terminology is consistent with that commonly
used for the Navier-Stokes system). On the other hand, even in 2D,
we expect that for ``weak solutions'' defect measures would appear in
\eqref{calore} due to the fact that the \rhs\ is only controlled in $L^1$.
This is the reason that led the authors of \cite{ERS2}
to postulate additional regularity on the initial data in order
to get a stronger and more satisfactory concept of solution. Indeed, in \cite{ERS2}
the following result was proved:
\bete\label{teo:old}
 Let us assume\/ \eqref{hp:F1} (with $C^2$ in place of $C^3$),
 \eqref{hp:F2}, \eqref{hp:F3} and\/ \eqref{hp:kappa}. Let also $T>0$
 and let $z_0=(\vu_0,\fhi_0,\teta_0)\in \calH$ additionally satisfy
 \begin{equation} \label{hp:ers2}
   z_0 \in V \times H^3(\Omega)\times V.
 \end{equation}
 Moreover, assume that
 \begin{equation} \label{hp:basso}
   \esiste \tetagiu>0 ~~\text{such that }\teta_0(x) \ge \tetagiu~~\text{a.e.~in }\Omega.
 \end{equation}
 Then, there exists at least
 one\/ {\rm ``strong solution''} to the non-isothermal model for two-phase
 fluid flows, namely, one quadruple $(\vu,\fhi,\mu,\vt)$ with
 \begin{align} \label{rego:vu}
    & \vu \in H^1(0,T;H) \cap L^\infty(0,T;V) \cap L^2(0,T;H^2(\Omega)),\\
  \label{rego:fhi}
    & \fhi \in W^{1,\infty}(0,T;V') \cap H^1(0,T;V) \cap L^2(0,T;H^3(\Omega)),\\
   \label{rego:mu}
    & \mu \in H^1(0,T;V') \cap L^\infty(0,T;V) \cap L^2(0,T;H^3(\Omega)),\\
   \label{rego:vt}
    & \vt \in H^1(0,T;V') \cap L^\infty(0,T;L^{q+2}(\Omega)) \cap L^2(0,T;V),
      \quad \vt>0~\hbox{a.e. in }(0,T)\times \Omega,\\
   \label{rego:Kvt}
    & K(\vt) \in L^2(0,T;V),
 \end{align}
 such that the equations of the system\/ \eqref{incom}-\eqref{CH2}
 hold in the sense of distributions as well as almost everywhere
 in~$(0,T)\times\Omega$, while~\eqref{calore} holds in $V'$ for
 a.e.~$t\in(0,T)$. Moreover, the following initial conditions
 hold a.e.~in~$\Omega$:
 \begin{equation} \label{iniz:old}
   \vu|_{t=0} = \vu_0, \quad
   \fhi|_{t=0} = \fhi_0, \quad
   \teta|_{t=0} = \teta_0.
 \end{equation}
\ente
\noindent%
\beos\label{rem:quad}
 In Theorem \ref{teo:old} we have noted the solution as a quadruple $(\vu,\fhi,\mu,\teta)$.
 However, in view of the fact that $\mu$ can be regarded as an auxiliary variable,
 in some situations it will be convenient to ``exclude''
 $\mu$ from the definition and interpret the solution just as a triple $z$, namely setting
 $z := (\vu,\fhi,\teta)$. Indeed, one can easily rewrite the system \eqref{CH1}-\eqref{CH2}
 as a single equation where $\mu$ no longer appears.
 This interpretation is particularly useful when we consider the
 dynamical system associated with (stable) solutions. Indeed, here $z$ is the
 natural variable and the set $\calV^r$ can be seen as a {\sl phase space}\/ for solution trajectories.
 On the other hand, the situation is somehow articulated, because
 we will see that, for determining $\omega$-limit sets, also the limit value of $\mu$ will
 play a specific role.
\eddos
\noindent%
It is worth discussing a bit more the hypotheses
on initial data  considered  in the above theorem.
First of all, \eqref{hp:basso} (i.e., the fact that the initial temperature is
assumed to be uniformly strictly positive) was just taken for simplicity,
whereas in fact a weaker assumption suffices for the proof.
On the contrary, conditions
\eqref{hp:ers2} were essential. These correspond in fact to the regularity
of ``strong solutions'' for the Navier-Stokes system (available in 2D)
and to the so-called ``second energy estimate'' for the Cahn-Hilliard equation.
Correspondingly, an improvement of the regularity of $\teta_0$ is also required.

Conditions \eqref{hp:ers2} give rise to a new functional setting suitable
for ``strong solutions''. Namely, one may define
\begin{equation}\label{defi:V}
  \mathcal V := \big\{ z=(\vu,\fhi,\teta)\in  \mathcal{H}  \cap (V \times H^3(\Omega)\times V) \big\}.
\end{equation}
In this notation, existence was proved in \cite{ERS2} for initial data
$z_0\in \calV$ additionally satisfying \eqref{hp:basso}. On the other hand,
one immediately sees from the statement of Theorem~\ref{teo:old}
that from $z_0\in\calV$ does not seem to follow that $z(t)$
is controlled in $\calV$ uniformly in time.
In other words, strong solutions appear to exhibit some regularity loss,
or, in the terminology of dynamical systems, $\calV$ seems not to be
a good phase space. In addition to that, it is also not clarified
whether property \eqref{hp:basso} is preserved in the evolution.

The main reason for this regularity gap is probably due to the use of limit-case
two dimensional embeddings in the a-priori estimates. In particular,
properties \eqref{rego:vu}, \eqref{rego:mu}  imply that the \rhs\ of \eqref{calore}
lies {\sl exactly}\/ in $L^2(0,T;H)$ and this information
seems not sufficient in order to get
any additional regularity on $\teta$.
In particular, an $L^\infty$-bound is lacking (Moser iterations do
not work for $L^2$ \rhs), which would be crucial in order to manage
some coefficients of \eqref{calore} which grow like powers of $\teta$.
In particular, the possibility to control $\teta$ in the space $V$ uniformly
in time is tied to the a-priori estimate obtained testing \eqref{calore}
by $\teta_t$ (or, more naturally, by $K(\teta)_t$). Without a
previous $L^\infty$-control of $\teta$, this procedure does not seem to
be available.

On the basis of these considerations, in order to avoid the regularity gap
occurring in Theorem~\ref{teo:old}, we decided to consider a concept
of solution which is {\sl slightly}\/ more regular with respect to
``strong solutions''. There are probably several ways to do this;
in our approach we will ``fractionally'' improve the regularity asked
for the initial velocity in such a way to eventually get a control
of the \rhs\ of \eqref{calore} in $L^p$ for some $p$ {\sl strictly}\/
greater than $2$. Hence, for $r\in(0,1/2]$ assigned but otherwise
arbitrary, we introduce the functional class
\begin{equation} \label{spaziofasi}
  \mathcal V^r:=
   \big\{ z=(\vu,\fhi,\teta)\in \calH \cap (H^{r+1}(\Omega)\times H^3(\Omega)\times V):
      K(\vt)\in V,
     ~ 1/\teta\in L^1(\Omega) \big\}.
\end{equation}
Compared to \eqref{hp:ers2}, we have made a number of changes, which it
is worth explaining in some detail. First of all,
we have improved a bit the regularity asked on $\vu_0$; secondly,
we have made precise the positivity condition on the temperature, namely,
we have replaced \eqref{hp:basso} with the much weaker condition $1/\teta\in L^1(\Omega)$.
This property is also more natural in view of the fact that it
is easy to show (see Lemma~\ref{thetaalpha} below) that it is preserved in the time evolution; i.e.,
if it holds for the initial datum, then it keeps holding with time.
Finally, in place of $\vt\in V$ we required $K(\vt)\in V$. Actually, both conditions
are stated in the definition; however, it is easy to check that $K(\vt)\in V$ implies
$\vt\in V$, but the converse is not true because $V$-functions in 2D
are not necessarily bounded in the uniform norm.

In the sequel we will actually
prove that if the initial datum lie in $\calV^r$ then the evolution
takes place in $\calV^r$ for any $t>0$ (and, also, one has a uniform in time
control of the ``magnitude'' of the solution in $\calV$).
In other words there is no regularity loss, which justifies our
choice to denote as ``stable solutions'' those solutions that start
from initial data in $\calV^r$.
Leaving to the next section the detailed presentation
of our main results, here we just recall that the set $\calV^r$
is also a metric space with respect to the distance
\begin{equation} \label{distVr}
  \dist_{\calV^r} (z_1,z_2)
    := \|\vu_1-\vu_2\|_{H^{r+1}(\Omega)}
     + \|\fhi_1-\fhi_2\|_{H^{3}(\Omega)}
     + \|K(\teta_1)-K(\teta_2)\|_{V}
     + \| 1/\teta_1 - 1/\teta_2 \|_{L^1(\Omega)}.
\end{equation}
This distance will be later used in order to properly settle the
(dissipative) dynamical process associated with stable solutions.


\section{Main results}
\label{sec:main}

We are now ready to present our main results. Our first theorem
is devoted to proving well-posedness of system \eqref{incom}-\eqref{calore}
in the space $\calV^r$ on finite time intervals:
\bete\label{teo:main}
 Let us assume\/ \eqref{hp:F1}-\eqref{hp:F4} and \eqref{hp:kappa}. Let also $T>0$.
 Then, given any $z_0\in \calV^r$, there exists {\sl a unique}\/ {\rm stable solution}
 to our problem, namely a quadruple $(\vu,\fhi,\mu,\teta)$
 with the regularity
 \begin{align} \label{rego:vu:new}
    & \vu \in H^1(0,T;H^r(\Omega)) \cap L^\infty(0,T;H^{1+r}(\Omega)) \cap L^2(0,T;H^{2+r}(\Omega)),\\
  \label{rego:fhi:new}
    & \fhi \in W^{1,\infty}(0,T;V') \cap H^1(0,T;V) \cap L^\infty(0,T;H^3(\Omega)),\\
  \label{rego:mu:new}
  & \mu \in H^1(0,T;V') \cap L^\infty(0,T;V) \cap L^2(0,T;H^3(\Omega)), \\
   \label{rego:vt:new}
    & \vt \in H^1(0,T;H) \cap L^\infty(0,T;V) \cap L^2(0,T;H^2(\Omega)),  \quad \vt>0~~\hbox{a.e. in }\,(0,T)\times \Omega,\\
      \label{rego:Kvt:new}
    & K(\vt) \in L^\infty(0,T;V)\cap L^2(0,T;H^2(\Omega)),\\
    \label{rego:1vt:new}
    &    1/\teta\in L^\infty(0,T;L^1(\Omega)),
 \end{align}
 satisfying equations\/ \eqref{incom}-\eqref{calore} a.e.~in $\Omega\times (0,T)$ and complying
 with the initial conditions
 \begin{equation} \label{iniz:new}
   \vu|_{t=0} = \vu_0, \quad
   \fhi|_{t=0} = \fhi_0, \quad
   \teta|_{t=0} = \teta_0
 \end{equation}
 almost everywhere in $\Omega$.
\ente
\noindent
Once we have well-posedness on finite time intervals, we can consider the dynamical
process associated with stable solutions. To this aim we need to introduce a further
(more regular) functional set:
\begin{equation}\label{spaziocpt}
  \mathcal W := \big\{ z=(\vu,\fhi,\teta)\in \calH \cap (H^{2}(\Omega)\times H^4(\Omega)\times H^2(\Omega)):~
   1/\teta\in L^{4}(\Omega),~\nabla(1/\teta)\in L^1(\Omega) \big\}.
\end{equation}
It is a standard matter to verify that $\mathcal W$ is embedded continuously
and compactly into $\calV^r$. As above, $\calW$ is a
(complete) metric space with the distance
\begin{align} \nonumber
  & \dist_{\calW} (z_1,z_2)
    := \|\vu_1-\vu_2\|_{H^{2}(\Omega)}
     + \|\fhi_1-\fhi_2\|_{H^{4}(\Omega)}
    + \|\teta_1-\teta_2\|_{H^2(\Omega)} \\
 \label{distW}
 & \mbox{}~~~~~
    + \| 1/\teta_1 - 1/\teta_2 \|_{L^4(\Omega)}
    + \| \nabla(1/\teta_1) - \nabla(1/\teta_2) \|_{L^1(\Omega)}.
\end{align}
In the sequel, in order to estimate the ``magnitude'' of the elements of
$\calW$ we will often write
\begin{equation} \label{norW}
  \|z\|_{\mathcal W} := \|\vu\|_{H^2(\Omega)}
   + \|\fhi\|_{H^4(\Omega)}
   + \|\teta\|_{H^2(\Omega)}
   + \|1/\teta\|_{L^{4}(\Omega)}
   + \|\nabla (1/{\teta})\|_{L^1(\Omega)}.
\end{equation}
This is of course somehow an abuse of notation because the above is not a norm.
Mimicking \eqref{norW}, we will also write
\begin{equation} \label{norH}
  \|z\|_{\mathcal H} := \|\vu\|
   + \|\fhi\|_{V}
   + \|\teta\|_{L^1(\Omega)}
   + \| \log \teta \|_{L^1(\Omega)},
\end{equation}
as well as $\|z\|_{\mathcal V}$,
$\|z\|_{\mathcal V^r}$, etc., with obvious corresponding notation.
\beos\label{rem:Moser}
 One may wonder where does the $L^4$-regularity in \eqref{distW} (and \eqref{norW})
 come out. Actually, this is somehow an arbitrary choice (taken just for simplicity),
 as it is also essentially arbitrary the $L^1$-condition $1/\teta$ in the definition of
 $\calV^r$. Indeed, at the price of technicalities, we expect that one
 could prove that, in general, if $\teta_0>0$ a.e.~in $\Omega$ and there exists
 (an arbitrarily small) $\alpha>0$ such that $\teta_0^{-\alpha}\in L^1(\Omega)$,
 then it turns out that $\teta^{-1}(t)\in L^\infty(\Omega)$ for any $t>0$.
 In other words, the temperature should become strictly positive instantaneously
 in the uniform norm. The more natural way to prove this form of the minimum principle argument
 would exploit Lemma~\ref{thetaalpha} below together with a Moser iteration procedure.
 However, the argument may involve quite a relevant amount of technicalities
 (cf., e.g., \cite{SSZ} for a similar situation) and we omit it because
 it is not essential for our purposes.
\eddos
\noindent%
Of course, in order to analyze the long-time behavior of solutions we
need to consider the natural problem constraints, therefore we set
\begin{defin}\label{spaceLambda}
 Given $m,M\in\RR$ and $\bm\in \RR^2$,
 we introduce the spaces $\mathcal{H}_{\bm, m, M}$,
 $\mathcal{V}_{\bm, m, M}$, $\mathcal{V}^r_{\bm, m, M}$,
 $\mathcal{W}_{\bm, m, M}$ of triplets $z = (\vu, \varphi, \vt)$
 respectively in $\mathcal{H}$, $\mathcal{V}$, $\mathcal{V}^r$,
 $\mathcal{W}$, subject to the constraints
 \begin{align}\label{vincolo1}
   & \vu\OO = \bm, \\
  \label{vincolo2}
   & \fhi\OO = m, \\
  \label{vincolo3}
   & \frac12 \| \vu \|^2
    + \frac12 \| \nabla \fhi \|^2
    + \io \big( F(\fhi) + \vt \big) = M.
 \end{align}
 and endowed with the corresponding distances introduced above.
\end{defin}
\noindent%
Given any triplet $(\bm, m, M)$, if $z_0=(\vu_0,\fhi_0,\teta_0)\in \mathcal{V}^r_{\bm, m, M}$,
then by the conservation properties \eqref{cons_medie_energia} and the regularity provided by
Theorem~\ref{teo:main}, the evolution of the system operates as a trajectory in the
space $\mathcal{V}_{\mathbf{m}, m, M}$. Hence, we can introduce the solution operator
\begin{align*}
  & S(t): \mathcal{V}^r_{\bm, m, M} \rightarrow \mathcal{V}^r_{\bm, m, M}, \qquad t \ge 0,\\[1mm]
  & z_0 \mapsto S(t)z_0 = (\vu(t), \varphi(t), \vt(t)).
\end{align*}
\beos\label{rem:semig}
 One can easily prove that
 $S$ is a semigroup on $\mathcal{V}_{\bm, m, M}^r$, i.e.,
 \begin{align*}
   & S(0) = {\rm Id},\\[1mm]
   & S(t + s) = S(t)S(s), \quad \forall s,t \ge 0.
 \end{align*}
 More precisely, using interpolation it is not difficult to deduce from \eqref{rego:vu:new}-\eqref{rego:1vt:new} that
 $S$ is a continuous semigroup. Namely, trajectories are continuous with values in $\calV^r$, moreover,
 \[
   S(t): \mathcal{V}_{\mathbf{m}, m, M}^r \to \mathcal{V}_{\mathbf{m}, m, M}^r
 \]
 is a continuous mapping for every $t\ge 0$.
\eddos
\noindent%
Section~\ref{quattro} will be devoted to the study of the asymptotic behaviour of our model system.
In particular, we will prove existence of nonempty $\omega$-limit sets of trajectories
in the case when the spatial mean of the initial velocity is $\bzero$.
In particular, this will permit us to prove our main result regarding the
long-time behavior of the system, i.e., existence of the
global attractor for the dynamical process associated to stable solutions.

Before stating it, we need however to introduce a further subclass of solutions. Namely,
for given $R\in\RR$, we set $\calV^{r,R}:=\{z\in \calV^r:~( - \log \vt)\OO \le R\}$
as well as the subclasses $\calV^{r,R}_{\bm,m,M}$, $\calV^{r,R}_{\bzero,m,M}$.
The space $\calV^{r,R}$ can be easily proved to be a closed (hence complete)
metric subspace of $\calV^r$. Observing that for the present model the entropy
density is given by $\log\vt$, $\calV^{r,R}$ turns out to contain
the configurations for which the global entropy 
is greater or equal than $-R$. Referring to the monographs
\cite{Hale,Rob,SY,Temam} for the basic concepts and definitions
from the theory of infinite-dimensional dynamical systems,
we can prove existence of the global attractor for solutions taking values in
the phase space $\calV^{r,R}_{\bzero,m,M}$.
\bete \label{teo:att}
 Assume that\/ \eqref{hp:F1}-\eqref{hp:F4} and\/ \eqref{hp:kappa} hold true.
 Let also $m,M,R\in \mathbb R$. Then, the space $\calV^{r,R}_{\bzero,m,M}$ is
 positively invariant for the semigroup $S(t)$. Moreover, there exists
 the global attractor $\calA=\calA^R(\mathbf{0}, m,M)$.
 Namely, $\calA$ is a compact subset of $\calV^{r,R}_{\mathbf 0, m,M}$ and
 is completely invariant for the flow $S(t)$ (i.e., $S(t)\calA = \calA$
 for all $t\ge 0$). Moreover, $\calA$ uniformly attracts the trajectory
 bundles starting from any bounded set $B\subset \calV^{r,R}_{\mathbf 0, m,M}$,
 i.e., we have
 \begin{equation} \label{attr}
   \lim_{t\nearrow \infty}
    \Dist_{\calV^r}(S(t)B,\calA) = 0,
 \end{equation}
 where $\Dist_{\calV^r}$ denotes the\/ Hausdorff semidistance associated
 to the metric of $\calV^r$, namely, if $E,F$ are subsets of $\calV^r$,
 we have set
 \[
   \Dist_{\calV^r}(E,F):=\sup_{e\in E} \inf_{f\in F}
    \dist_{\calV^r}(e,f).
 \]
\ente
\noindent%
The reason for restricting ourselves to the case when the spatial mean of the initial velocity is
$\mathbf{0}$ stands in the fact that, for $\bm \neq \mathbf{0}$, solution trajectories
asymptotically tend to rotate around the flat torus with constant speed $\bm$, which makes
the long-time analysis more difficult. We will give more details
in Section~\ref{quattro} where in particular we will see that a description of the
long-time behavior of trajectories in the case $\bm\neq\bzero$ can be provided by
means of a suitable change of variables. In the proof we will also explain
the reason why we need to restrict ourselves to the subclass
$\calV^{r,R}_{\bzero,m,M}$, namely, why we need to impose a lower bound to
the spatial mean of the entropy.


\section{Proof of Theorem \ref{teo:main} part I: existence of solutions with additional regularity}
\label{sec:rego}

The proof will be carried out by performing a number of {\sl formal}\/ a-priori
estimates holding for any hypothetical solution to our system. These estimates
could be made rigorous by adapting them to the regularization argument
given in~\cite{ERS2}. In the sequel we shall denote by $c$ a generic
positive constant depending only on the assigned parameters of
the system (but independent of the time variable).
In an expression like $Q(a,b,\dots)$, $Q$ will denote a generic (computable)
positive function increasingly monotone in each of its arguments.
For instance, $Q(\| z_0 \|_{\calV^r})$ represents a computable quantity
depending only of the fixed parameters of the system, increasingly
depending on the $\calV^r$-``magnitude'' of the initial datum,
and independent of the time variable.


\subsection{Preliminaries and technical results}

We start by presenting a technical lemma that provides an estimate of negative
powers of the temperature on time intervals of arbitrary length.
\begin{lemm}
\label{thetaalpha}
Let $z=(\vu,\fhi,\teta)$ be a solution to our system defined over the
generic time interval $(S,T)$ with $0\le S < T \le +\infty$
and satisfying, for some $\alpha>1$,
\begin{equation} \label{nablate}
  N:= \| \nabla \teta \|_{L^2(S,T;H)} < \infty, \qquad\quad
   A:= \| \teta^{1-\alpha}(S) \|_{L^1(\Omega)} < \infty.
\end{equation}
Then, we have
\[
  \teta^{1 - \alpha} \in L^{\infty}(S,T; L^1(\Omega)),
   \qquad \nabla (\teta^{\frac{1 - \alpha}{2}}) \in L^2(S,T; L^2(\Omega)),		
\]
with the quantitative estimate
\begin{equation} \label{quantimin}
  \| \teta^{1 - \alpha} \|_{ L^\infty (S,T; L^1(\Omega)) }
   + \| \nabla (\teta^{\frac{1 - \alpha}{2}} ) \|_{L^2(S,T; H)}
   \le Q(N,A).
\end{equation}
\end{lemm}
\begin{proof}
We did not specify in the statement the concept of weak solutions
we are dealing with because, in fact, what we are proving is a structural
property of any solution of the system (even of
very weak ones, provided they comply with the energy conservation
and entropy production principle at least in a regularization).

That said, we multiply \eqref{calore} by $-\teta^{-\alpha}$, for some $\alpha>1$ getting
\begin{align*}
\ddt \io \dfrac{\teta^{1-\alpha}}{\alpha-1}-\io \teta^{1-\alpha}\Delta\mu+\io\Delta[K(\teta)]\teta^{-\alpha}+
\io (|\nabla \vu|^2+|\nabla\mu|^2)\teta^{-\alpha}=0,
\end{align*}
where the third term on the left hand side reads as
$$
\io\Delta[K(\teta)]\teta^{-\alpha}=
\dfrac{4\alpha}{(1-\alpha)^2}\|\nabla(\teta^{\frac{1-\alpha}{2}})\|^2+
\dfrac{4\alpha}{(q+1-\alpha)^2}\|\nabla(\teta^{\frac{1-\alpha+q}{2}})\|^2,
$$
while the second one is
\begin{align*}
-\io \teta^{1-\alpha}\Delta\mu & =
(1-\alpha)\io \teta^{-\alpha}\nabla\teta \nabla\mu\geq
-\dfrac12 \io \teta^{-\alpha}|\nabla\mu|^2-\dfrac{(1-\alpha)^2}2\io\teta^{-\alpha}|\nabla\teta|^2\\
& \geq
-\dfrac12 \io \teta^{-\alpha}|\nabla\mu|^2-\dfrac{2\alpha}{(1-\alpha)^2}
   \io|\nabla(\teta^{\frac{1-\alpha}2})|^2-c\io |\nabla\teta|^2
\end{align*}
having observed that, due to Young's inequality,
$$
\io \teta^{-\alpha}|\nabla\teta|^2\leq \delta \io \teta^{-1-\alpha}|\nabla\teta|^2+c_\delta\io |\nabla\teta|^2,
$$
for any fixed $\delta>0$.
Collecting the above estimates, we are lead to
\begin{align*}
&\ddt \io \dfrac{\teta^{1-\alpha}}{\alpha-1}+
\dfrac{2\alpha}{(1-\alpha)^2}\|\nabla(\teta^{\frac{1-\alpha}{2}})\|^2
+\dfrac{4\alpha}{(q+1-\alpha)^2}\|\nabla(\teta^{\frac{1-\alpha+q}{2}})\|^2\\
 & \mbox{}~~~~~
  + \io (|\nabla \vu|^2+\dfrac12|\nabla\mu|^2)\teta^{-\alpha}\leq c\|\nabla \teta\|^2.\nonumber
\end{align*}
Integrating this inequality over $(S,t)$, $t\in(S,T]$, we then deduce
\begin{equation}\label{alfetta1}
  \frac{1}{\alpha - 1} \io\teta^{1-\alpha}(t)
   +\dfrac{2\alpha}{(1-\alpha)^2}\int_S^t \|\nabla(\teta^{\frac{1-\alpha}{2}})\|^2 \,\dis
    \leq \frac{1}{\alpha - 1} \io\teta^{1-\alpha}(S) + c \int_S^t \|\nabla\teta(s)\|^2 \,\dis.
\end{equation}
Taking the essential supremum as $t$ varies in $(S,T]$, recalling
\eqref{nablate}, we get the assert.
\end{proof}


\subsection{Higher regularity - part I}
\label{part1}

Here we will derive some estimates holding on the assigned bounded interval $(0,T)$ of finite
length, whereas in Section \ref{quattro} we will look for uniform in time estimates.
For this reason in this section we will allow the generic (positive) constants
$C$ to depend on $T$. The value of $C$ may vary on occurrence.

Hence, to start the proof of Theorem~\ref{teo:main}, we take a ``strong solution''
$z=(\vu,\fhi,\mu,\teta)$,
as provided by Theorem \ref{teo:old}, and we will show that, under our
slightly stronger assumptions (particularly on the initial datum
$z_0\in\calV^r$), $z$ is in fact a ``stable solution''.
To this aim, we will start proving the following regularity properties:
\begin{itemize}
 \item[(i)]  $\vu\in L^\infty(0,T;H^{r+1}(\Omega)) \cap L^2(0,T;H^{r+2}(\Omega))$,
 \item[(ii)] $\teta\in H^1(0,T; H) \cap L^\infty(0,T;V)$ and  $K(\teta)\in L^\infty(0,T;V)$,
 \item[(iii)] $\fhi\in L^\infty(0,T;H^3(\Omega))$,
 \item[(iv)] $1/\teta\in L^\infty(0,T;L^1(\Omega))$.
\end{itemize}

\noindent
$\bullet \,$ {\sc proof of} (i). First of all, we observe that
\begin{equation}
\label{rego:2:fhi}
\varphi \in L^{\infty}(0,T; H^2(\Omega)).
\end{equation}
This simply comes from \eqref{CH2} using \eqref{rego:mu}, \eqref{rego:vt} and \eqref{hp:F3} combined with the fact that,
due to \eqref{rego:fhi} $\varphi \in L^{\infty}(0,T; L^p(\Omega))$ for all $p > 1$. This estimate will be used several times later on.

To achieve (i) we need now to introduce the Stokes operator (for more details see for instance \cite[Par. 3.8]{SY} or
\cite[Par. 2.2]{Temam83}) $A: \mathbb{V} \rightarrow \mathbb{V}'$, where $A \vu := \mathbb{P} (- \Delta \vu)$
for $\vu \in D(A) = \{\mathbf{v} \in \mathbb{H}, \Delta \mathbf{v} \in \mathbb{H}\}$ and
$\mathbb{P}$ denotes the so called Leray (or Helmholtz) projection. It can be shown
(cf.~\cite[Theorem 38.6]{SY}) that the operator $A$  is an unbounded, positive, linear, selfadjoint operator on the space $\mathbb{H}$.
Therefore, we can define the powers $A^{s}$, $s \in \mathbb{R}$, with domain $D(A^{s})$ in $\mathbb{H}$. Setting
\[
\mathbb{V}^{s} = D(A^{s/2}),
\]
then $\mathbb{V}^{s}$ is a closed subspace of $H^{s}_0(\Omega)$ (here denoting the subspace of $H^s(\Omega)$
containing the function with zero spatial mean) and indeed
\[
\mathbb{V}^{s} = \{\mathbf{v} \in H^{s}_0(\Omega): \dive \mathbf{v} = 0\}.
\]
In particular $\mathbb{V}^2 = D(A)$, $\mathbb{V}^1 = \mathbb{V}$, $\mathbb{V}^0 = \mathbb{H}$. Moreover $A$ is an
isomorphism from $\mathbb{V}^{s + 2}$ onto $\mathbb{V}^{s}$, from $D(A)$ onto $\mathbb{H}$, from $\mathbb{V}$ onto
$\mathbb{V}'$ and so on. Finally, the norm $\|A^{s/2} \mathbf{v}\|$ on $\mathbb{V}^{s}$ is equivalent to the norm induced by $H^{s}_0(\Omega)$.

The subsequent step will be to test \eqref{ns} by $A^{r+1} \hat\vu$, where $r$ is as in \eqref{spaziofasi} and where
$\hat \vu = \vu - {\vu}_{\Omega}$. Working in divergence-free spaces, the key point will be to project equation \eqref{ns}
into the space $\mathbb{H}$ by applying $\mathbb{P}$ and then to test by $A^{r + 1} \hat\vu$. Let us note that $\mathbb{P} \vu = \vu$
due to \eqref{incom}; moreover also $\mathbb{P} \vu_t = \vu_t$. The pressure term is missing because $\mathbb{P} \nabla p = 0$,
whence  we are left with the following term:
\[
\mathbb{P} (\vu \cdot \nabla \vu) = \mathbb{P}(\vu \cdot \nabla \hat \vu) = \mathbb{P}(\hat \vu \cdot \nabla \hat \vu + {\vu}_{\Omega} \cdot \nabla \hat \vu) = \mathbb{P} (\hat \vu \cdot \nabla \hat \vu).
\]
Next, we set
\begin{equation}
\label{projf}
f := \mathbb{P}(- \dive (\nabla \varphi \otimes \nabla \varphi))
\end{equation}
and define the trilinear form
\[
b(\bu, \mathbf{v}, \mathbf{w}) := \langle (\bu \cdot \nabla) \mathbf{v}, \mathbf{w} \rangle.
\]
As a consequence, we get
\begin{equation}
\label{eq:Ar}
\frac{1}{2} \ddt \|\hat{\vu}\|^2_{\mathbb{V}^{r+1}}
 + \|\hat\bu\|^2_{\mathbb{V}^{r+2}} = - b (\hat\bu, \hat\bu, A^{r+1} \hat\bu) + \langle A^{\frac{r}{2}} f, A^{\frac{r}{2} + 1} \hat \bu \rangle
\end{equation}
and we observe that $f_{\Omega} = 0$. If $r \in (0,1)$ and $f \in \mathbb{V}^r$, then \cite[Part I, Sec. 2.3, formula (2.20)]{Temam83}
\begin{equation}
\label{rego:Temam}
\|A^{\frac{r}{2}} f\| = \|f\|_{\mathbb{V}^r} \le \, \|f\|^{1-r} \|f\|^r_{\mathbb{V}} \le \, \|f\|^{1-r} \|\nabla f\|^r.
\end{equation}
Now, using \eqref{dis:L4}, it is not difficult to show that
\[
   \|f\| \le C \|\nabla^2 \varphi\|_{L^4(\Omega)} \, \|\nabla \varphi\|_{L^4(\Omega)}
  \le C \|\varphi\|^{1/2}_{H^3(\Omega)} \|\varphi\|_{H^2(\Omega)} \|\varphi\|^{1/2}_V \stackrel{\eqref{rego:2:fhi}}{\le} \, C \, \|\varphi\|_{H^3(\Omega)}^{1/2}.
\]
This, together with Poincar\'e's inequality, entails
\begin{align}
   \|\nabla f\|
    & = \left \|\nabla \left (\nabla \varphi \Delta \varphi + \nabla \left (\frac{|\nabla \varphi|^2}{2} \right ) \right ) \right \| \nonumber\\
    & = \left \|\nabla^2 \varphi \Delta \varphi + \nabla \varphi \otimes \nabla \Delta \varphi + \nabla^2 \left (\frac{|\nabla \varphi|^2}{2} \right ) \right \|\nonumber\\
  & \le \left (\|\nabla^2 \varphi \|_{L^4(\Omega)} \|\Delta \varphi\|_{L^4(\Omega)}
      + \|\nabla \varphi\|_{L^{\infty}(\Omega)} \|\nabla \Delta \varphi\|
    + \left \|\nabla^2 \left (\frac{|\nabla \varphi|^2}{2} \right ) \right \|  \right )\nonumber\\
  & \stackrel{\eqref{dis:L4}, \eqref{dis:Linfty}}{\le} C \, \left (\|\varphi\|_{H^2(\Omega)} \|\varphi\|_{H^3(\Omega)}
    + \|\nabla \varphi\|^{1/2} \|\varphi\|^{\frac{3}{2}}_{H^3(\Omega)} + \|\nabla \varphi\|_{L^{\infty}(\Omega)} \|\varphi\|_{H^3(\Omega)}
     \right ) \nonumber\\
   &\stackrel{\eqref{dis:Linfty}, \eqref{rego:2:fhi}}{\le} \, C \, \left (\|\varphi\|^{\frac{3}{2}}_{H^3(\Omega)} + 1 \right ) \label{rego:Temam:2}.
\end{align}
This implies, going back to \eqref{rego:Temam},
\begin{equation}
\label{stimadifVr}
\|f\|_{\mathbb{V}^r} \le \, C \, \|\varphi\|^{\frac{1-r}{2}}_{H^3(\Omega)} \left (\|\varphi\|_{H^3(\Omega)}^{\frac{3}{2} r}  + 1 \right ) \le \, C \, \left (\|\varphi\|_{H^3(\Omega)}^{\frac{1 + 2 r}{2}} + 1 \right )
\end{equation}
and in turn
\[
\langle A^{\frac{r}{2}} f, A^{\frac{r}{2} + 1} \hat \bu \rangle 
\le C \, \|f\|_{\mathbb{V}^r} \|\hat \bu\|_{\mathbb{V}^{r+2}}  
\le \frac{1}{4} \|\hat \bu\|^2_{\mathbb{V}^{r+2}} + C \, \|f\|^2_{\mathbb{V}^r} 
  \le  \frac{1}{4} \|\hat \bu\|^2_{\mathbb{V}^{r + 2}} + C \left (\|\varphi\|_{H^3(\Omega)}^{2 r + 1} + 1 \right ).
\]
Next, we have
\begin{equation}
\label{bhatu}
   b(\hat \bu, \hat \bu, A^{r+1} \hat \bu) = \langle \hat \bu \cdot \nabla \hat \bu, A^{r+1} \hat \bu \rangle
  = \langle A^{\frac{r}{2}}(\hat \bu \cdot \nabla \hat \bu), A^{\frac{r}{2} + 1} \hat \bu \rangle \le \|\hat \bu \cdot \nabla \hat \bu\|_{\mathbb{V}^r} \, \|\hat \bu\|_{\mathbb{V}^{r+2}}.
\end{equation}
Now, also on account of the Friedrichs inequality \eqref{Friedr},
\begin{align*}
  \|\hat{\bu} \cdot \nabla \hat \bu\|_{\mathbb{V}^r}
    & \le \|\hat{\bu} \cdot \nabla \hat \bu\|^{1-r} \, \|\hat{\bu} \cdot \nabla \hat \bu\|_{\mathbb{V}}^{r} \\
    & \le \, C \, \|\hat \bu\|_{L^{\infty}(\Omega)}^{1-r} \, \|\nabla \hat \bu\|^{1-r} \left (\|\nabla \hat \bu\|_{L^4(\Omega)}^{2r}
    + \| \hat \bu\|_{L^{\infty}(\Omega)}^r \, \|\nabla^2 \hat \bu\|^r \right ) =: V_1 + V_2.
\end{align*}
We have that
\[
  V_1 \stackrel{\eqref{dis:Linfty}}{\le} \, C \|\hat \bu\|^{\frac{1-r}{2}} \|\hat \bu\|_{\mathbb{V}^2}^{\frac{1-r}{2}} \, \|\nabla \hat \bu\|^{1 - r} \, \|\nabla \hat \bu\|^r \, \|\nabla \hat \bu\|^r_{\mathbb{V}}
  \le \, C \, \|\hat \bu\|^{\frac{r+1}{2}}_{\mathbb{V}^2}  \le \, C \, \|\hat \bu\|_{\mathbb{V}^2},
\]
where we used \eqref{rego:vu} and the fact that $0 < r \le 1/2$. Analogously
\[
  V_2 \le \, \|\hat \bu\|_{L^{\infty}(\Omega)} \|\hat \bu\|^{1-r}_{\mathbb{V}} \|\hat \bu\|_{\mathbb{V}^2}^r
     \stackrel{\eqref{dis:Linfty}}{\le} \|\hat \bu\|^{\frac{1}{2} + r}_{\mathbb{V}^2} \, \|\hat \bu\|^{\frac{1}{2}} \, \|\hat \bu\|_{\mathbb{V}}^{1-r}
     \le \, C \, \|\hat \bu\|_{\mathbb{V}^2}^{\frac{1+2r}{2}} \le \, C \, \|\hat \bu\|_{\mathbb{V}^2}.
\]
Summing up, we deduce
\begin{equation} \label{Bone}
  \|\hat{\bu} \cdot \nabla \hat \bu\|_{\mathbb{V}^r} \le  \, C \, \|\hat \bu\|_{\mathbb{V}^2}.
\end{equation}
This fact, using interpolation together with \eqref{bhatu}, entails that
\begin{align*}
  b(\hat \bu, \hat \bu, A^{r+1} \hat \bu) & \le \|\hat \bu \cdot \nabla \hat \bu\|_{\mathbb{V}^r} \, \|\hat \bu\|_{\mathbb{V}^{r+2}}
    \le \, C \, \|\hat \bu\|_{\mathbb{V}^2} \|\hat \bu\|_{\mathbb{V}^{r+2}}\\
    & \le \, C \, \|\hat \bu\|_{\mathbb{V}^{r+1}}^{r} \|\hat \bu\|_{\mathbb{V}^{r+2}}^{2 - r}\le \, \frac{1}{4} \|\hat \bu\|^2_{\mathbb{V}^{r+2}} + C \, \|\hat \bu\|^2_{\mathbb{V}^{r+1}}.
\end{align*}
Coming back to \eqref{eq:Ar}, we finally obtain
\begin{equation}
\label{fineprimaparte}
   \ddt \|\hat{\vu}\|^2_{\mathbb{V}^{r+1}}
 + \|\hat\bu\|^2_{\mathbb{V}^{r+2}}
    \le \, C \, \left (\|\varphi\|_{H^3(\Omega)}^{1 + 2 r} + 1 \right ) + C \, \|\hat \bu\|^2_{\mathbb{V}^{r+1}}
\end{equation}
and the right hand side is summable since $0 < r \le \frac{1}{2}$. Then, (i) follows from
Gronwall's lemma. It is also worth remarking that (i) yields in particular
$\vu\in L^\infty(0,T;L^\infty(\Omega))$.

\smallskip

\noindent $\bullet \,$ {\sc proof of} (ii). We can now address (ii), passing through
the intermediate step
\begin{equation}
\label{rego:vt:step3}
\vt \in {L^{\infty}(0,T; L^p(\Omega))} \qquad \qquad \textnormal{for all $p > 1$.}
\end{equation}
In order to show this property,
we multiply \eqref{calore} by $\vt^p$ and integrate over $\Omega$. We remark that
\[
  \int_{\Omega} ( \vu \cdot \nabla \vt) \vt^p = 0
\]
due to \eqref{incom} and the choice of periodic boundary conditions, whence we have
\begin{align}
 & \frac{1}{p+1} \ddt \int_{\Omega} \vt^{p+1}
  + \frac{4p}{(p+1)^2} \int_{\Omega} \left |\nabla \vt^{\frac{p+1}{2}} \right |^2
  + \frac{4p}{(p+q+1)^2}  \int_{\Omega} \left |\nabla \vt^{\frac{p+q+1}{2}}\right |^2 \label{step1.1}\\
 & \mbox{}~~~~~\le  \int_{\Omega} |\Delta \mu| \vt^{p+1}
  + \int_{\Omega} |\nabla \bu|^2 \vt^p
  + \int_{\Omega} |\nabla \mu|^2 \vt^p. \nonumber
\end{align}
At this point, from \eqref{step1.1} and \eqref{poinc}, we deduce
\begin{align}\nonumber
 & \ddt \int_{\Omega} \vt^{p+1}
  + \frac{4p}{c_p(p+1)}  \left \|\vt^{\frac{p+1}{2}} \right \|^2_V \\
 & \mbox{}~~~~~\le \frac{4p}{(p+1)} \|\vt\|^{p+1}_{L^1(\Omega)}
   + (p+1) \int_{\Omega} |\Delta \mu| \vt^{p+1}
   + (p+1)\int_{\Omega} |\nabla \bu|^2 \vt^p
   + (p+1) \int_{\Omega} |\nabla \mu|^2 \vt^p \nonumber\\
 & \mbox{}~~~~~\le (p+1) \int_{\Omega} (|\Delta \mu| + 1) \vt^{p+1}
   + (p+1) \int_{\Omega} (|\nabla \bu|^2 + |\nabla \mu|^2) \vt^p
   =: I_a + I_b, \label{step1.2}
\end{align}
having observed that $\frac{4p}{p+1} \le p+1$ and where $c_p$ is the Poincar\'e constant in \eqref{poinc}. Next,
\begin{align*}
  I_a &:= (p+1) \int_{\Omega} (|\Delta \mu| + 1) \vt^{p+1}
    = (p+1) \int_{\Omega} \vt^{\frac{p+1}{2}}(|\Delta \mu| + 1) \vt^{\frac{p+1}{2}}\\
  &\le c(p) \left \|\vt^{\frac{p+1}{2}} \right\|_V \left\|\vt^{\frac{p+1}{2}} (|\Delta \mu| + 1)\right\|_{V'}
    \le \frac{2p}{c_p(p+1)} \left\|\vt^{\frac{p+1}{2}}\right\|^2_V + c(p) \left\|\vt^{\frac{p+1}{2}} (|\Delta \mu| + 1) \right\|^2_{L^{6/5}(\Omega)},
\end{align*}
where the exponent $6/5$ is taken just for computational convenience and where from now on $c(p)$ denotes a positive constant depending on $p$, possibly varying from line to line.
Now, applying H\"older's inequality with exponents $\frac{5}{3}$ and $\frac{5}{2}$, we obtain
\[
  \left\|\vt^{\frac{p+1}{2}} (|\Delta \mu| + 1)\right\|^2_{L^{6/5}(\Omega)}
   = \left (\int_{\Omega} \left [\vt^{\frac{p+1}{2}} (|\Delta \mu| + 1)\right]^{\frac{6}{5}} \right )^{\frac{5}{3}}
    \le \||\Delta \mu| + 1\|_{L^3(\Omega)}^2 \left  ( \int_{\Omega} \vt^{p+1} \right ).
\]
Thus we end up with
\begin{equation}
\label{stima:I}
  I_a :=  (p+1) \int_{\Omega} (|\Delta \mu| + 1) \vt^{p+1}
   \le \frac{2p}{c_p (p+1)} \left\|\vt^{\frac{p+1}{2}}\right\|^2_V
    + c(p) \||\Delta \mu| + 1\|_{L^3(\Omega)}^2 \left  ( \int_{\Omega} \vt^{p+1} \right ).
\end{equation}
On the other hand, by H\"older's inequality with exponents $p+1$ and $\frac{p+1}{p}$ we deduce
\begin{align*}
  I_b & := (p+1) \int_{\Omega} (|\nabla \bu|^2 + |\nabla \mu|^2) \vt^p
   \le c(p) \left (\|\nabla \bu\|^2_{L^{2(p+1)}(\Omega)}
      +  \|\nabla \mu\|^2_{L^{2(p+1)}(\Omega)} \right ) \left (\int_{\Omega} \vt^{p+1} + 1 \right )\\
  &\le  c(p) \left (\|\bu\|^2_{H^2(\Omega)} + \|\mu\|^2_{H^2(\Omega)} \right ) \left (\int_{\Omega} \vt^{p+1} + 1 \right ).
\end{align*}
Summing up, we obtain
\[
  \ddt \int_{\Omega} \vt^{p+1}
  \le c(p) \left (\||\Delta \mu| + 1\|^2_{L^3(\Omega)}
        + \|\bu\|^2_{H^2(\Omega)}
        + \|\mu\|^2_{H^2(\Omega)} \right ) \left (\int_{\Omega} \vt^{p+1} + 1 \right ).
\]
Then, the conclusion comes from the Gronwall inequality (see for instance \cite[Lemma 2.8]{Rob}) by
recalling \eqref{rego:mu}, \eqref{rego:vu} and \eqref{hp:basso}.

\smallskip

\noindent%
We can now address the proof of $\teta\in L^\infty(0,T;V)$.
To this aim, we formally multiply \eqref{calore}
by $\partial_t K(\vt) = \kappa(\vt) \vt_t$. We deduce
\begin{align*}
  & \int_{\Omega} |\sqrt{\kappa(\vt)} \vt_t|^2
   + \frac{1}{2} \ddt \int_{\Omega} |\nabla K(\vt)|^2  \\
  & \mbox{}~~~~~= - \int_{\Omega} \bu \cdot \nabla \vt \, \kappa(\vt) \vt_t
     - \int_{\Omega} \vt \Delta\mu \kappa(\vt) \vt_t
     + \int_{\Omega} |\nabla \bu|^2 \kappa(\vt) \vt_t
     + \int_{\Omega} |\nabla \mu|^2 \kappa(\vt) \vt_t\\
   & \mbox{}~~~~~  =: I + II + III + IV.
\end{align*}
First of all, we have
\begin{align*}
  I & := - \int_{\Omega} \bu \cdot \nabla \vt \, \kappa(\vt) \vt_t
   \le \int_{\Omega} |\bu| |\nabla K(\vt)| |\vt_t|
   \le C \|\bu\|_{L^{\infty}(\Omega)} \|\nabla K(\vt)\| \|\vt_t\|\\
  & \le \frac{1}{4} \|\sqrt{\kappa(\vt)} \vt_t\|^2
    + C \|\bu\|^2_{L^{\infty}(\Omega)} \|\nabla K(\vt)\|^2
   \le \frac{1}{4} \|\sqrt{\kappa(\vt)} \vt_t\|^2
    + C \|\nabla K(\vt)\|^2
 %
 %
\end{align*}
where we used the fact that $|\vt_t| \le |\sqrt{\kappa(\vt)} \vt_t|$ because $\kappa(\vt) \ge 1$
and the fact that $\bu$ is bounded in the uniform norm as a consequence of~(i) and of
the continuous embedding $H^{1+r}(\Omega)\subset L^\infty(\Omega)$.

Let us now deal with the term $II$. We have
\begin{align*}
  II &:= - \int_{\Omega} \vt \Delta \mu \kappa(\vt) \vt_t
   \le \int_{\Omega} \vt \sqrt{\kappa(\vt)} |\Delta \mu| \sqrt{\kappa(\vt)} |\vt_t|\\
  &\le \frac{1}{4} \|\sqrt{\kappa(\vt)} \vt_t\|^2
   + C \|\Delta \mu\|^2_{L^4(\Omega)} \|\sqrt{\kappa(\vt)} \vt\|^2_{L^4(\Omega)}\\
  & \stackrel{\eqref{dis:L4}, \eqref{rego:vt:step3}, \eqref{rego:mu}}{\le}
    \frac{1}{4} \|\sqrt{\kappa(\vt)} \vt_t\|^2 + C \|\Delta \mu\|_V^2.
\end{align*}
The most difficult term to be estimated is $III$: indeed,
as remarked above, we are lacking the information
$\vt \in L^{\infty}((0,T) \times \Omega)$
(we only have \eqref{rego:vt:step3} at this level)
and for this reason we need to use the improved regularity on $\vu$
in the following way:
\begin{align*}
  III & := \int_{\Omega} |\nabla \bu|^2 \sqrt{\kappa(\vt)} \sqrt{\kappa(\vt)} \vt_t  \\
  & \le C \|\nabla \bu\|^2_{L^{2s}(\Omega)} \|\sqrt{\kappa(\vt)}\|_{L^{s'}(\Omega)} \|\sqrt{\kappa(\vt)} \vt_t\|\\
 &\le \frac{1}{4} \|\sqrt{\kappa(\vt)}\vt_t\|^2 +
  C \|\nabla \bu\|^4_{L^{2s}(\Omega)},
\end{align*}
for some exponents $s, s' > 2$ such that $\frac{1}{s} + \frac{1}{s'} = \frac{1}{2}$;
we will specify these exponents below. Here we used \eqref{rego:vt:step3}.
At this point we look for a suitable exponent $s$ such that
\[
  \int_0^T \|\nabla \bu(t)\|^4_{L^{2s}(\Omega)} \, \dit \le C.
\]
First of all we apply \eqref{dis:Brezis} with the choices $r := 2s$ and $v := \nabla \bu$. We have
\begin{equation}
\label{stimaL2s}
   \|\nabla \bu\|_{L^{2s}(\Omega)}^4
    \le C  \|\nabla \bu\|^{\frac{4}{s}} \|\nabla \bu\|_{H^1(\Omega)}^{4 \left (1 - \frac{1}{s} \right )}.
\end{equation}
At this point we use \eqref{dis:BrezziGilardi} with the choices $s = 1$, $s_1 = 0$, $s_2 = 1 + r$. We obtain
\[
  \|\nabla \bu\|_{H^1(\Omega)}
   \le C \|\nabla \bu\|^{\frac{r}{r+1}} \|\nabla \bu\|_{H^{r+1}(\Omega)}^{\frac{1}{r+1}}.
\]
Hence, combining the latter estimate with \eqref{stimaL2s}, we deduce
\begin{align*}
  \int_0^T \|\nabla \bu\|_{L^{2s}(\Omega)}^4
   & \le  C \int_0^T \|\nabla \bu\|^{\frac{4}{s}} \|\nabla \bu\|_{H^1(\Omega)}^{4 \left (1 - \frac{1}{s} \right )} \\
  & \le C \int_0^T \|\nabla \bu\|^{\frac{4}{s}} \left [ \|\nabla \bu\|^{\frac{r}{r+1}} \|\nabla \bu\|_{H^{r+1}(\Omega)}^{\frac{1}{r+1}} \right ]^{4 \left (1 - \frac{1}{s} \right )} \\
  & \le C \int_0^T \|\nabla \bu\|^{\frac{4}{s} + \frac{4r(s-1)}{(r+1)s}} \|\nabla \bu\|^{\frac{4 (s-1)}{s(r+1)}}_{H^{r+1}(\Omega)}.
\end{align*}
The first term in the last product is controlled uniformly in time since $\nabla \bu \in L^{\infty}(0,T; L^2(\Omega))$.
Hence the integral is bounded provided (see~(i))
\[
  \frac{4 (s-1)}{s(r+1)} \le 2 \Leftrightarrow s \le \frac{2}{1-r}.
\]
We recall that $0 < r \le 1/2$. So for instance if $r = 1/2$ then we can choose $s = 4$.

Finally, the estimate of the term $IV$ can be done in the same way as we did for the term $III$
(note that $\nabla \mu$ has even more regularity than $\nabla \bu$). We then conclude by the Gronwall
inequality that
\begin{equation}
 \label{rego:vt:new:2}
   \vt \in H^1(0,T; L^2(\Omega)) \qquad \qquad K(\vt) \in L^{\infty}(0,T; V),
\end{equation}
so that (ii) holds true. Note that at this point it is crucial to
assume $K(\vt_0)\in V$ in place of the sole property $\vt_0\in V$ considered
in \cite{ERS2}.

\smallskip

\noindent%
$\bullet \,$ {\sc proof of} (iii).
Now, $\fhi\in L^\infty(0,T;H^3(\Omega))$, that is $\rm (iii)$,
follows easily: indeed, reading \eqref{CH2} as the elliptic equation
$-\Delta\fhi+F'(\fhi)=\teta+\mu\in L^\infty(0,T;V)$,
we readily obtain the desired conclusion.

\smallskip

\noindent%
$\bullet \,$ {\sc proof of} (iv).
This is directly achieved from Lemma \ref{thetaalpha} applied
over the interval $(0,T)$ with $\alpha = 2$. Note that we use
here \eqref{rego:vt} and the assumption $(1/\teta_0) \in L^1(\Omega)$
resulting from the choice of $z_0\in \calV^r$.


\subsection{Higher regularity - part II}
\label{part2}

We will now show that, under the assumptions of Theorem~\ref{teo:main},
the following additional regularity properties hold:
\begin{itemize}
 \item[{\rm(v)}] $\vu \in H^1(0,T; H^r(\Omega))$,
 \item[{\rm (vi)}] $\vt \in \LDHD$, $K(\vt) \in \LDHD$.
\end{itemize}
Actually, this completes the proof that $z$ is a stable solution.

\smallskip

\noindent%
$\bullet \,$ {\sc proof of} (v).
To prove the additional regularity for $\vu$, we proceed as in the proof of (i) in section \ref{part1}: we first project
equation \eqref{ns} into the space $\mathbb{H}$ by applying the operator $\mathbb{P}$, then we test by $A^r \vu_t$.
Setting $\hat{\vu} = \vu - \vu_{\Omega}$ (so that $\hat\vu_t=\vu_t$)
and recalling \eqref{projf}, we obtain
\begin{align*}
   \|\hat{\vu}_t\|^2_{\mathbb{V}^r}
    + \frac{1}{2} \ddt \|\hat{\vu}\|_{\mathbb{V}^{r+1}}^2
   = - b(\hat{\vu}, \hat{\vu}, A^r \hat{\vu}_t)
    + \langle A^{r/2} f, A^{r/2} \hat{\vu}_t \rangle.
\end{align*}
We can estimate the trilinear term working as in \eqref{bhatu}, namely we have
\begin{equation}
  - b(\hat{\vu}, \hat \bu, A^r \hat{\vu}_t)
    \le C \|\hat \bu \cdot \nabla\hat \bu\|_{\mathbb{V}^r} \|\hat \bu_t\|_{\mathbb{V}^{r}}
    \stackrel{\eqref{Bone}}{\le} C \|\hat \bu\|_{\mathbb{V}^{2}} \|\hat \bu_t\|_{\mathbb{V}^{r}}
    \le \frac{1}{4} \|\hat \bu_t\|_{\mathbb{V}^{r}}^2
      + C \|\hat \bu\|_{\mathbb{V}^{2}}^2.\label{stimadiA3}
\end{equation}
On the other hand,
\begin{align}
  \langle A^{r/2} f, A^{r/2} \hat{\vu}_t \rangle
   & \le \|f\|_{\mathbb{V}^r} \, \|\hat{\vu}_t\|_{\mathbb{V}^r} \label{stimadiA4} \\
   & \stackrel{\eqref{stimadifVr}}{\le} C \left (\|\varphi\|_{H^3(\Omega)}^{\frac{1 + 2r}{2}} + 1 \right )
       \|\hat{\vu}_t\|_{\mathbb{V}^r}
    \le \frac{1}{4} \|\hat{\vu}_t\|^2_{\mathbb{V}^r} + C \left (\|\varphi\|_{H^3(\Omega)}^{1 + 2r} + 1 \right ). \nonumber
\end{align}
Summing up, we finally obtain (we use here the fact that $r \le 1/2$)
\[
  \|\hat{\vu}_t\|^2_{\mathbb{V}^r} + \ddt \|\hat{\vu}\|^2_{\mathbb{V}^{r+1}}
    \le C \left (\|\varphi\|_{H^3(\Omega)}^2 + 1 \right )
      + C \|\hat{\vu}\|^2_{\mathbb{V}^{2}},
\]
and the conclusion comes by \eqref{rego:vu} and \eqref{rego:fhi}.

\smallskip

\noindent%
$\bullet \,$ {\sc proof of} (vi).
Comparing terms in equation \eqref{calore} and using in particular
\eqref{rego:vt:new} together with the already observed fact that
$\bu$ is bounded in the uniform norm,
we easily deduce
\begin{equation}\label{rego:DeltaK}
  \Delta K(\vt) \in L^2(0,T; L^2(\Omega)),
\end{equation}
whence the second of~(vi).
Now we would like to conclude that $\vt \in \LDHD$. Actually,
using the second \eqref{rego:vt:new} and interpolation, we first notice
that
\begin{equation*}
  K(\vt) \in L^4(0,T;W^{1,4}(\Omega)),
\end{equation*}
which, in view of \eqref{defiK}), is equivalent to
\begin{equation}\label{st:x}
  \vt + \vt^{q+1} \in L^4(0,T;W^{1,4}(\Omega)).
\end{equation}
We now observe that
$$
  - \Delta K(\vt)
   = - \kappa (\vt) \Delta \vt - \kappa'(\vt) | \nabla \vt |^2.
$$
Therefore,
$$
  \iTo ( - \kappa (\vt) \Delta \vt - \kappa'(\vt) | \nabla \vt |^2 )^2
   = \| - \Delta K(\vt) \|_{L^2(0,T; L^2(\Omega))}^2\le C.
$$
Let us deal with the integral in the left hand side. We have
$$
  \iTo ( - \kappa (\vt) \Delta \vt - \kappa'(\vt) | \nabla \vt |^2 )^2
   =: \iTo \kappa^2(\vt) |\Delta \vt|^2
    + \iTo (\kappa'(\vt))^2 | \nabla \vt |^4
  + J,
$$
where
\begin{align*}
  J & := 2 \iTo \kappa(\vt) \kappa'(\vt) \Delta \vt | \nabla \vt |^2 \\
  & \le \frac12 \iTo \kappa^2(\vt) | \Delta \vt |^2
   + C \iTo (\kappa'(\vt))^2 | \nabla \vt |^4 \\
  & = \frac12 \iTo \kappa^2(\vt) | \Delta \vt |^2
   + C \iTo \vt^{2q-2} | \nabla \vt |^4,
\end{align*}
and the last term in the right hand side can be rewritten as
$$
  C \iTo | \nabla \vt^{\frac{q+1}2} |^4,
$$
which is controlled by \eqref{st:x} (in fact this implies that we can control all the
intermediate powers of $\vt$ between $1$ and $q+1$). Hence, using the fact that
$\kappa(\vt) \ge 1$ and applying elliptic regularity,  we finally come to the goal
\begin{equation}
\label{vtLinftyH2}
  \vt \in \LDHD.
\end{equation}


\section{Proof of Theorem \ref{teo:main} part II: uniqueness}
\label{sec:uniq}

We now prove the uniqueness part of Theorem~\ref{teo:main}. Let then $z_0\in \mathcal V^r$,
and let $(\bu_i, \fhi_i,\mu_i,\th_i)$, $i=1,2$, be a couple of stable solutions both emanating
from $z_0$ over the interval $(0,T)$. Setting
$(\bu, \fhi,\mu,\th):=(\bu_1-\bu_2,\fhi_1-\fhi_2, \mu_1-\mu_2,\vt_1-\vt_2)$,
it is readily seen that, then,
\begin{align}
& \div  \bu=0
\label{incomd}\\
&  \bu_t+\bu_1\cdot \nabla \bu+\bu\cdot \nabla \bu_2=\Delta \bu-\div(\nabla\fhi_1\otimes\nabla \fhi)-\div(\nabla \fhi\otimes\nabla\fhi_2)
\label{nsd}\\
& \fhi_t+\bu_1\cdot \nabla\fhi+\bu\cdot \nabla\fhi_2=\Delta \mu
\label{CH1d}\\
&
\mu =-\Delta\fhi+F'(\fhi_1)-F'(\fhi_2)-\th
\label{CH2d}\\
& \th_t+\bu_1\cdot \nabla \th+\bu\cdot \nabla \vt_2+\vt_1\Delta  \mu+\th\Delta  \mu_2-\Delta[K(\vt_1)-K(\vt_2)]
\label{calored}\\
& \mbox{}~~~~~ =
(\nabla \bu_1+\nabla \bu_2)\cdot \nabla \bu+(\nabla \mu_1+\nabla \mu_2)\cdot \nabla  \mu\nonumber
\end{align}
supplemented with null initial data. This guarantees for instance that $\varphi_{\Omega}(t) = 0$ and $\bu_\Omega(t)=\mathbf 0$ for all $t \ge 0$.

Taking advantage of the regularity properties \eqref{rego:vu:new}-\eqref{rego:Kvt:new}, we then
have (actually even something more is true)
\begin{equation}
\label{inizio}
\begin{cases}
\|\bu_i(t)\|_{V}+\|\fhi_i(t)\|_{H^3(\Omega)}+\|\mu_i(t)\|_V+\|\vt_i(t)\|_{V}\leq c,\quad t\in (0,T)\\
\|\bu_i\|_{L^2(0,T;H^2(\Omega))}+\|\mu_i\|_{L^2(0,T;H^3(\Omega))}+\|\vt_i\|_{L^2(0,T;H^2(\Omega))}\leq c
\end{cases}
\end{equation}
for some positive constant $c$ depending on $T$ and on the initial data. We will use the above properties
repeatedly in the sequel. The proof is actually based on the combination of several estimates,
presented in separate subsections for more clarity.


\subsection{Preliminary estimates}

Having in mind to test \eqref{CH1d} by ${\mathcal N}\fhi_t$, we
first estimate $\nabla  \mu$ multiplying \eqref{CH1d} by $ \mu-\mu_\Omega$. This gives
\begin{align*}
\|\nabla \mu\|^2&
=-\l \fhi_t, \mu-\mu_\Omega\r-\l \bu_1\cdot \nabla\fhi, \mu-\mu_\Omega\r-\l \bu\cdot \nabla \fhi_2, \mu-\mu_\Omega\r\\
& \leq
\|\fhi_t\|_{V'}\| \mu-\mu_\Omega\|_V+\|\bu_1\|_{L^3(\Omega)}\|\nabla \fhi\| \| \mu-\mu_\Omega\|_{L^6(\Omega)}+
\|\bu\|\|\nabla \fhi_2\|_{L^4(\Omega)}\| \mu-\mu_\Omega\|_{L^4(\Omega)}.
\end{align*}
Then Friedrich's inequality and \eqref{inizio} yield
\begin{equation}
\label{diffmu}
\|\nabla \mu\|^2\leq c(\|\fhi_t\|^2_{V'}+\|\nabla\fhi\|^2+\| \bu\|^2).
\end{equation}
Next, we control $\Delta\fhi$ by taking the product of \eqref{CH2d} with $-\Delta\fhi$:
\begin{align*}
 \|\Delta \fhi\|^2& =-\l \mu-\mu_\Omega,\Delta\fhi\r-\l \th-\tho,\Delta\fhi\r+\l F'(\fhi_1)-F'(\fhi_2),\Delta\fhi\r\\
  & \leq c \|\Delta\fhi\|(\|\nabla  \mu\|+\|\th-\tho\|+\|F'(\fhi_1)-F'(\fhi_2)\|)
\end{align*}
Recalling \eqref{hp:F4} and exploiting H\"older's inequality with exponents $3$ and $3/2$ together with \eqref{inizio},
we deduce
\begin{equation}
\label{diffeffeH}
\|F'(\fhi_1)-F'(\fhi_2)\|^2\leq c\l (1+|\fhi_1|^{2p_F}+|\fhi_2|^{2p_F}),\fhi^2\r\leq c \|\fhi\|^2_{L^3(\Omega)}\leq c \|\fhi\|^2_{V}.
\end{equation}
Thus, also on account of \eqref{diffmu} we conclude that
\begin{equation}
\label{diffDeltaph}
\|\Delta\fhi\|^2\leq c(\|\nabla  \mu\|^2+\|\th-\tho\|^2+\|\fhi\|^2_{V})\leq
c(\| \bu\|^2+\|\fhi_t\|^2_{V'}+\|\th-\tho\|^2+\|\fhi\|^2_{V}).
\end{equation}


\subsection{Difference of fluid velocities}

The product of \eqref{nsd} by $\bu$, on account of \eqref{incomd}, leads to
\begin{align}
\label{diffu}
 \dfrac12 \ddt \|\bu\|^2+\|\nabla \bu\|^2=-\l \bu\cdot \nabla \bu_2,\bu\r+\l \nabla \fhi_1\otimes \nabla \fhi,\nabla \bu\r+\l
   \nabla \fhi\otimes \nabla\fhi_2,\nabla \bu\r,
\end{align}
where the right hand side is easily controlled thanks to \eqref{inizio}. Indeed, by \eqref{dis:L4} and \eqref{Friedr},
\begin{align*}
\text{RHS} & \leq \|\nabla \bu_2\| \|\bu\|^2_{L^4(\Omega)}+(\|\nabla \fhi_1\|_{L^\infty(\Omega)} +\|\nabla \fhi_2\|_{L^\infty(\Omega)} )\|\nabla \fhi\| \|\nabla \bu\|\\
& \leq \dfrac 12 \|\nabla \bu\|^2+ c\big[\|\bu\|^2+(\|\fhi_1\|_{H^3(\Omega)}+\|\fhi_2\|_{H^3(\Omega)})\|\nabla \fhi\|^2\big].
\end{align*}
Therefore,
\begin{align}
\label{diffu1}
 \ddt \|\bu\|^2+\|\nabla \bu\|^2 \stackrel{\eqref{inizio}}{\leq}  c (\|\bu\|^2 +\|\nabla \fhi\|^2).
\end{align}


\subsection{Difference of temperatures' means}

Integrating \eqref{calored} over $\Omega$ we obtain
\begin{equation}
\label{mediad}
|\Omega|( {{\tho}})_t=\l\nabla \vt_1,\nabla  \mu\r-\l \th-\tho,\Delta\mu_2\r+
\l (\nabla \bu_1+\nabla \bu_2),\nabla \bu\r+
\l (\nabla \mu_1+\nabla \mu_2),\nabla  \mu\r,
\end{equation}
which, by \eqref{inizio}, entails
\begin{equation}
\label{boundmean}
|(\tho)_t|\leq
c(\|\nabla  \mu\|+\|\mu_2\|_{H^3(\Omega)}\|\th-\tho\|_{V'}+\|\nabla \bu\|),
\end{equation}
Moreover, the product of \eqref{mediad} by $\tho$ gives, for (small)
$\alpha,\beta>0$ to be chosen later and corresponding (large) $c>0$,
\begin{align}
\label{add-4}
  \dfrac12 \ddt \tho^2 & \leq c|\tho|(\|\nabla  \mu\|+\|\mu_2\|_{H^3(\Omega)}\|\th-\tho\|_{V'}+\|\nabla \bu\|)\\
   & \leq \beta \|\nabla \bu\|^2+\alpha \eps \|\nabla  \mu\|^2+c(\tho^2+\|\mu_2\|^2_{H^3(\Omega)}\|\th-\tho\|^2_{V'}).
\nonumber
\end{align}


\subsection{Difference of order parameters}

Multiplying
\eqref{CH1d} by ${\mathcal N}\fhi_t$,
we obtain
\begin{equation*}
 \|\fhi_t\|^2_{V'}+\l  \mu,\fhi_t\r=
 -\l \bu_1\cdot \nabla \fhi,{\mathcal N}\fhi_t\r-\l \bu\cdot \nabla \fhi_2,{\mathcal N}\fhi_t\r,
\end{equation*}
where, taking the product of \eqref{CH2d} by $\fhi_t$, the second term reads
\begin{align*}
\l  \mu,\fhi_t\r & = \dfrac 12 \ddt \left( \|\nabla \fhi\|^2-2\l \th-\tho,\fhi\r\right)
+\l F'(\fhi_1)-F'(\fhi_2),\fhi_t\r
+\l \th_t,\fhi\r.
\end{align*}
Combining the above relations, we then deduce
\begin{align}\label{g:11}
  & \|\fhi_t\|^2_{V'}+\dfrac 12 \ddt( \|\nabla \fhi\|^2-2\l \th-\tho,\fhi\r)\\
 \nonumber
  & \mbox{}~~~~~
   =-\l F'(\fhi_1)-F'(\fhi_2),\fhi_t\r
    -\l \bu_1\cdot \nabla \fhi,{\mathcal N}\fhi_t\r-\l \bu\cdot \nabla \fhi_2,{\mathcal N}\fhi_t\r -\l \th_t,\fhi\r.
\end{align}
In order to estimate the fourth term on the right hand side, we multiply
 \eqref{calored} by $\fhi$:
\begin{align*}
 \l \th_t, \fhi\r & =
 \l K(\vt_1)-K(\vt_2),\Delta\fhi\r
 +\l \bu_1(\th-\tho),\nabla \fhi\r-\l \bu\cdot \nabla \vt_2,\fhi\r
 +\l\nabla  \mu,\nabla \vt_1 \, \fhi\r\\
 & \quad
 + \l\nabla  \mu, \vt_1\nabla \fhi\r
 -\l \th\Delta\mu_2,\fhi\r+\l (\nabla \bu_1+\nabla \bu_2)\cdot \nabla \bu+(\nabla \mu_1+\nabla \mu_2)\cdot  \nabla  \mu,\fhi\r.
 \end{align*}
Using \eqref{CH2d}, the first term on the right hand side gives
\begin{align*}
  & \l K(\vt_1)-K(\vt_2),\Delta\fhi\r \\
  & \mbox{}~~~~~ =\l  K(\vt_1)-K(\vt_2),-\mu +\mu_\Omega+F'(\fhi_1)-F'(\fhi_2)-\th+\tho\r
   - (\mu_\Omega+\tho)\int_\Omega [K(\vt_1)-K(\vt_2)]\\
  & \mbox{}~~~~~ =\l  K(\vt_1)-K(\vt_2),-\mu +\mu_\Omega+F'(\fhi_1)-F'(\fhi_2) - (F'(\fhi_1)-F'(\fhi_2))_{\Omega}-\th+\tho\r\\
  & \mbox{}~~~~~ =\l  K(\vt_1)-K(\vt_2),-\mu +\mu_\Omega+F'(\fhi_1)-F'(\fhi_2) - (F'(\fhi_1)-F'(\fhi_2))_{\Omega}\r\\
  & \mbox{}~~~~~~~~~~~~~~~
  -\l  K(\vt_1)-K(\vt_2),\th-\tho\r.
\end{align*}
Hence,
\begin{align*}
  & \l \th_t, \fhi\r =-\l  K(\vt_1)-K(\vt_2),\mu -\mu_\Omega -F'(\fhi_1)+F'(\fhi_2) + (F'(\fhi_1)-F'(\fhi_2))_{\Omega}\r \\
  & \quad -  \l  K(\vt_1)-K(\vt_2),\th-\tho\r
  +\l \bu_1( \th- \tho),\nabla \fhi\r-\l \bu\cdot \nabla \vt_2,\fhi\r
  +\l\nabla  \mu,\nabla \vt_1\,\fhi\r \\
 & \quad
  + \l\nabla  \mu, \vt_1\nabla \fhi\r
  -\l \th\Delta\mu_2,\fhi\r
  + \l (\nabla \bu_1+\nabla \bu_2)\cdot \nabla \bu+(\nabla \mu_1+\nabla \mu_2)\cdot  \nabla  \mu,\fhi\r.
\end{align*}
Combining the above relation with \eqref{g:11}, we then deduce
\begin{align}\label{g:12}
  & \|\fhi_t\|^2_{V'}+\dfrac 12 \ddt( \|\nabla \fhi\|^2-2\l \th-\tho,\fhi\r)-\l  K(\vt_1)-K(\vt_2),\th-\tho\r\\
 \nonumber
  & \mbox{}\quad =-\l F'(\fhi_1)-F'(\fhi_2),\fhi_t\r
    + \l \bu_1 \fhi + \bu \fhi_2,\nabla{\mathcal N}\fhi_t\r\\
 \nonumber
  & \mbox{}\quad \quad +\l  K(\vt_1)-K(\vt_2),\mu -\mu_\Omega-F'(\fhi_1)+F'(\fhi_2) + (F'(\fhi_1)-F'(\fhi_2))_{\Omega}\r\\
 \nonumber
  & \mbox{}\quad \quad -\l \bu_1(\th-\tho),\nabla \fhi\r+\l \bu\cdot \nabla \vt_2,\fhi\r \\
 \nonumber
  & \mbox{}\quad \quad   -\l\nabla  \mu,\nabla \vt_1 \, \fhi\r  - \l\nabla  \mu, \vt_1\nabla \fhi\r
    +  \l \th\Delta\mu_2,\fhi\r  \\
 \nonumber
  & \mbox{}\quad \quad  -\l (\nabla \bu_1+\nabla \bu_2)\cdot \nabla \bu+(\nabla \mu_1+\nabla \mu_2)\cdot  \nabla  \mu,\fhi\r =: \mathcal I_1.
\end{align}
Owing to \eqref{hp:F4}, we have
\begin{equation}
\label{diffeffeV}
 \|\nabla (F'(\fhi_1)-F'(\fhi_2))\| \leq c\| |\nabla \fhi| (1+|\fhi_1|^{p_F})+|\fhi| (1+|\fhi_1|^{p_F-1}+|\fhi_2|^{p_F-1})|\nabla \fhi_2|\|\leq c\|\fhi\|_{V}.
\end{equation}
Thus, recalling that $\l \fhi_t, 1\r=0$, we obtain
\begin{align*}
-\l F'(\fhi_1)-F'(\fhi_2),\fhi_t\r & \leq C \|\nabla( F'(\fhi_1)-F'(\fhi_2))\|\|\fhi_t\|_{V'}
 \leq
c \|\fhi\|_V
\|\fhi_t\|_{V'}.
\end{align*}
Next, owing to \eqref{dis:L4} and \eqref{Friedr},
\begin{align*}
\l \bu_1 \fhi+ \bu \fhi_2,\nabla{\mathcal N}\fhi_t\r
& \leq c\|\fhi_t\|_{V'} ( \| \bu_1\|_{L^4(\Omega)} \|\fhi\|_{L^4(\Omega)} + \|\bu\| \|\fhi_2\|_{L^{\infty}(\Omega)})\\
& \leq c
\|\fhi_t\|_{V'}( \|\nabla \fhi\|+\|\bu\|).
\end{align*}
Having observed that
$$
  \|  K(\vt_1)-K(\vt_2)\|_{3/2}\leq c(\|\th - \tho\| + |\tho| )(1+\|\vt_1\|^q_V+\|\vt_2\|^q_V)\leq c(\|\th - \tho\| + |\tho|),
$$
thanks to \eqref{diffeffeV}, it is straightforward to obtain
\begin{align*}
 & \l  K(\vt_1)-K(\vt_2),\mu -\mu_\Omega-F'(\fhi_1)+F'(\fhi_2)+(F'(\fhi_1)-F'(\fhi_2))_\Omega\r\\
 & \leq
 \|  K(\vt_1)-K(\vt_2)\|_{L^{3/2}(\Omega)}\big(\|\mu -\mu_\Omega\|_{L^3(\Omega)} +\|F'(\fhi_1)-F'(\fhi_2)-(F'(\fhi_1)-F'(\fhi_2))_\Omega\|_{L^3(\Omega)}\big)\\
 & \leq
 c (\|\th - \tho\| + |\tho|) (\|\nabla  \mu\|+\|\fhi\|_V).
\end{align*}
Next, also on account of Agmon's inequality, we can estimate the remaining summands in
$\mathcal I_1$ as follows:
\begin{align*}
 & -\l \bu_1(\th-\tho),\nabla \fhi\r+\l \bu\cdot \nabla \vt_2,\fhi\r\\
 & \mbox{}\quad \leq
  \| \bu_1\|_{L^{\infty}(\Omega)} \|\th-\tho\| \|\nabla \fhi\|+\| \bu\| \| \nabla \vt_2\|_{L^4(\Omega)} \|\fhi\|_{L^4(\Omega)}\\
  &\mbox{}\quad \leq c\left( \| \bu_1\|^{1/2}_{H^2(\Omega)}\|\th-\tho\| +\| \bu\| \|\vt_2\|_{H^2(\Omega)}\right) \|\nabla \fhi\|;
\end{align*}
By \eqref{dis:L4}, \eqref{dis:Linfty} and \eqref{Friedr}
\begin{align*}
 -\l\nabla  \mu,\nabla \vt_1  \fhi+ \vt_1\nabla \fhi \,\r & \leq
 c\|\nabla  \mu\| \left(\|\nabla \vt_1\| \|\fhi\|_{L^{\infty}(\Omega)}+\|\nabla \fhi\|_{L^4(\Omega)}\|\vt_1\|_{L^4(\Omega)}\right)\\
   & \leq c\|\nabla  \mu\| \left( \|\fhi\|^{1/2} \|\fhi\|^{1/2}_{H^2(\Omega)}+\|\nabla \fhi\|^{1/2} \|\fhi\|^{1/2}_{H^2(\Omega)}\right)\\
  & \leq c\|\nabla  \mu\|  \|\nabla\fhi\|^{1/2} \|\fhi\|^{1/2}_{H^2(\Omega)};
\end{align*}
Exploiting \eqref{Friedr} and the injection $V\subset L^p(\Omega)$, for $p\geq 1$,
\begin{align*}
\l \th\Delta\mu_2,\fhi\r &
=\tho \l \Delta\mu_2,\fhi\r+\l (\th-\tho)\Delta\mu_2,\fhi\r\\
& \leq c |\tho| \|\nabla \mu_2\| \|\nabla\fhi\|+c \|\th-\tho\|\|\Delta\mu_2\|_{L^4(\Omega)} \|\fhi\|_{L^4(\Omega)}\\
& \leq c |\tho|  \|\nabla\fhi\|+c \|\th-\tho\|\|\mu_2\|_{H^3(\Omega)} \|\nabla \fhi\|;
\end{align*}
Finally, by the same token and \eqref{dis:L4},
\begin{align*}
  & -\l (\nabla \bu_1+\nabla \bu_2)\cdot \nabla \bu+(\nabla \mu_1+\nabla \mu_2)\cdot  \nabla  \mu,\fhi\r \\
  & \mbox{}\quad \leq
  (\|\nabla \bu_1+\nabla \bu_2\|_{L^4(\Omega)}\|\nabla \bu\| +\|\nabla \mu_1+\nabla \mu_2\|_{L^4(\Omega)}  \|\nabla  \mu\|)\|\fhi\|_{L^4(\Omega)}\\
 & \leq
  \big[ (\| \bu_1\|^{1/2}_{H^2(\Omega)}+\|\bu_2\|^{1/2}_{H^2(\Omega)})\|\nabla \bu\|
     +(\| \mu_1\|^{1/2}_{H^2(\Omega)}+\| \mu_2\|^{1/2}_{H^2(\Omega)})  \|\nabla  \mu\|\big] \|\nabla \fhi\|.
\end{align*}
Collecting the above computations, we finally arrive at
\begin{align*}
  \mathcal I_1 & \leq c(1+\| \bu_1\|^{1/2}_{H^2(\Omega)}+\|\bu_2\|^{1/2}_{H^2(\Omega)})\|\nabla \bu\| \|\nabla\fhi\|
    + c(1+\| \mu_1\|^{1/2}_{H^2(\Omega)}+\| \mu_2\|^{1/2}_{H^2(\Omega)})  \|\nabla  \mu\|\|\nabla \fhi\|\\
  & \quad +c\|\nabla  \mu\|  \|\nabla\fhi\|^{1/2} \|\fhi\|^{1/2}_{H^2(\Omega)}
    +c(\|\th - \tho\| + |\tho|) \|\nabla  \mu\| \\
  & \quad + c(\|\th - \tho\| + |\tho|) (1+\|\mu_2\|_{H^3(\Omega)})\|\fhi\|_V
  + c\|\fhi_t\|_{V'}(\|\fhi\|_V+\|\bu\|) \\
  & \quad + c\|\th-\tho\| \|\bu_1\|^{1/2}_{H^2(\Omega)}\|\nabla \fhi\|+c\|\vt_2\|_{H^2(\Omega)}\| \bu\|  \|\nabla \fhi\|\\
  & \leq \dfrac 12 \|\fhi_t\|^2_{V'}+\dfrac \alpha 2\|\nabla  \mu\|^2
    + \dfrac c\alpha \|\th-\tho\|^2+\dfrac \beta 2 \|\nabla \bu\|^2+\dfrac \alpha 2 \|\Delta \fhi\|^2+ g(t)(\| \fhi\|^2_V+\|\bu\|^2+\tho^2)\\
  & \leq \dfrac 12 \|\fhi_t\|^2_{V'}+\left( \dfrac \alpha 2+\alpha c\right)\|\nabla  \mu\|^2
    +\left( \dfrac c\alpha +\alpha c \right)\|\th-\tho\|^2+\dfrac \beta 2 \|\nabla \bu\|^2 + g(t)(\| \fhi\|^2_V+\|\bu\|^2+\tho^2).
\end{align*}
In the last passage we have used \eqref{CH2d} to control the term depending on $\Delta\fhi$, namely
we noted that
\begin{equation} \label{g0}
  \Delta \fhi = \Delta\fhi - (\Delta\fhi)\OO
   = F'(\fhi_1) - F'(\fhi_2) - ( F'(\fhi_1) - F'(\fhi_2) )\OO
   - (\vt-\tho) - (\mu - \mu\OO);
\end{equation}
moreover,  we have set
\begin{equation} \label{g}
   g(t):=c[1+\|\vt_2\|^2_{H^2(\Omega)}+\| \bu_1\|^2_{H^2(\Omega)}+\|\bu_2\|^2_{H^2(\Omega)}+\| \mu_1\|^2_{H^3(\Omega)}+\| \mu_2\|^2_{H^3(\Omega)}],
\end{equation}
where the (large) constant $c>0$ also depends on the choice of the (small) constants $\alpha,\beta>0$
(the letters $\alpha$ and $\beta$ were
already used before: we may take the smaller of the two choices for $\alpha,\beta$).
Collecting the above estimates, \eqref{g:12} finally gives
\begin{align}
\label{fifone}
  & \|\fhi_t\|^2_{V'}+ \ddt( \|\nabla \fhi\|^2-2\l \th-\tho,\fhi\r)-2\l  K(\vt_1)-K(\vt_2),\th-\tho\r\\
  &\leq \alpha (1+2c)\|\nabla  \mu\|^2
    +c\left( \dfrac 1\alpha+\alpha\right) \|\th-\tho\|^2+\beta\|\nabla \bu\|^2+ g(t)(\| \fhi\|^2_V+\|\bu\|^2+\tho^2).\nonumber
\end{align}
%


\subsection{Difference of temperatures}

Multiplying
\eqref{calored} by ${\mathcal N}(\th-\tho)$ and integrating by parts, we have
\begin{align*}
  & \dfrac 12 \ddt \|\th-\tho\|^2_{V'}+\l K(\vt_1)-K(\vt_2),\th-\tho\r\\
  & \quad = \l
   \bu_1(  \th-\tho), \nabla{\mathcal N}(\th-\tho)\r
   -\l \bu\cdot \nabla \vt_2,{\mathcal N}(\th-\tho)\r\\
  & \quad \quad  +\l \nabla \vt_1\cdot\nabla  \mu, {\mathcal N}(\th-\tho)\r
   +\l \vt_1 \nabla  \mu, \nabla {\mathcal N}(\th-\tho)\r\\
  & \quad\quad
   -\l (\th-\tho)\Delta  \mu_2,{\mathcal N}(\th-\tho)\r
   +\tho\l   \mu_2-(\mu_2)_\Omega,\th-\tho\r\\
  & \quad \quad +\l(\nabla \bu_1+\nabla \bu_2)\cdot \nabla \bu
   +(\nabla \mu_1+\nabla \mu_2)\cdot \nabla  \mu,{\mathcal N}(\th-\tho)\r
   =: I_1+I_2+I_3,
\end{align*}
where the terms $I_i$, $i=1,2,3$, will be specified below.
First of all, for any $\sigma>0$, by interpolation we have
\begin{align*}
  I_1 & := \l\bu_1(  \th-\tho)+ \vt_1 \nabla  \mu, \nabla {\mathcal N}(\th-\tho)\r\\
  & \leq c\|\bu_1\|_\infty \|\th-\tho\| \|\th-\tho\|_{V'}
    + c \|\vt_1\|_{L^{(4+2\sigma)/\sigma}(\Omega)} \|\nabla  \mu\| \|\nabla {\mathcal N}(\th-\tho)\|_{L^{2+\sigma}(\Omega)}\\
  &  \stackrel{\eqref{dis:Brezis}}{\leq}
     c\|\bu_1\|^{1/2}_{H^2(\Omega)} \|\th-\tho\| \|\th-\tho\|_{V'}
   + c \|\nabla  \mu\| \|\nabla {\mathcal N}(\th-\tho)\|^{2/(2+\sigma)} \|\nabla {\mathcal N}(\th-\tho)\|_{V}^{\sigma/(2+\sigma)}\\
  & \leq
     c\|\bu_1\|^{1/2}_{H^2(\Omega)} \|\th-\tho\| \|\th-\tho\|_{V'}
   + c \|\nabla  \mu\| \|\th-\tho\|_{V'}^{2/(2+\sigma)} \|\th-\tho\|^{\sigma/(2+\sigma)}.
\end{align*}
Next, using \eqref{dis:Linfty}, we notice that
$$
  \|{\mathcal N}(\th-\tho)\|_{L^{\infty}(\Omega)}\leq
   c\|\th-\tho\|^{1/2}_{V'} \|\th-\tho\|^{1/2}.
$$
Therefore,
\begin{align*}
  I_2 & := -\l \bu\cdot \nabla \vt_2+ \nabla \vt_1\cdot \nabla  \mu-(\th-\tho)\Delta  \mu_2,{\mathcal N}(\th-\tho)\r \\
  & \quad +\l(\nabla \bu_1+\nabla \bu_2)\cdot \nabla \bu
   +(\nabla \mu_1+\nabla \mu_2)\cdot \nabla  \mu,{\mathcal N}(\th-\tho)\r\\
  & \leq \|{\mathcal N}(\th-\tho)\|_{L^{\infty}(\Omega)} \big[ \|\vt_2\|_V \|\bu\|+\|\vt_1\|_V \|\nabla  \mu\|
     +\|\mu_2\|_{H^2(\Omega)}\|\th-\tho\|\\
  & \quad + (\|\nabla \bu_1\|+\|\nabla \bu_2\|)\|\nabla \bu\|+(\|\nabla \mu_2\|+\|\nabla \mu_2\|)\|\nabla  \mu\| \big] \\
  & \le   c\|\th-\tho\|^{1/2}_{V'} \|\th-\tho\|^{1/2} (\|\bu\|+\|\nabla \bu\|+\|\nabla  \mu\|)
   + c\|\mu_2\|_{H^2(\Omega)}\|\th-\tho\|^{1/2}_{V'} \|\th-\tho\|^{3/2}.
\end{align*}
Finally,
\begin{align*}
  I_3 & := \tho\l   \mu_2-(\mu_2)_\Omega,\th-\tho\r \leq
   c|\tho| \|\mu_2-(\mu_2)_\Omega\|_V\|\th-\tho\|_{V'}\leq c|\tho| \|\th-\tho\|_{V'}.
\end{align*}
Collecting the last three estimates and using Young's inequality, we eventually infer
\begin{align} \label{last}
  & \dfrac 12 \ddt \|\th-\tho\|^2_{V'}+\l K(\vt_1)-K(\vt_2),\th-\tho\r
    \leq \delta\eps  \|\th-\tho\|^2 \\
 \nonumber
  & \quad + \beta\|\nabla \bu\|^2+\alpha\eps\|\nabla\mu\|^2
   + C_* (1+\|\bu_1\|_{H^2(\Omega)}
   +\|\mu_2\|^2_{H^3(\Omega)})\|\th-\tho\|^2_{V'}+c(\tho^2+\|\bu\|^2),
\end{align}
where the ``large'' constant $C_*$ may depend on the ``small'' constants
$\alpha$, $\beta$, $\delta$, $\eps$ whose value will be specified
at the end.


\subsection{Conclusion}

In order to accomplish our purpose, we introduce the functional
\begin{align*}
  \ZZZ(t) & := 2\|\bu(t)\|^2
   +\eps\|\nabla \fhi(t)\|^2
   -2\eps \l \th(t)-\tho(t),\fhi(t)\r
   +\|\th(t)-\tho(t)\|^2_{V'}
   +\tho(t)^2,
\end{align*}
noticing that, provided $\eps>0$ is small enough,
$$
  \ZZZ(t)\geq c_\eps(\|\bu(t)\|^2+\|\nabla \fhi(t)\|^2+\|\th(t)-\tho(t)\|^2_{V'}+\tho(t)^2).
$$
Now, adding together \eqref{last} with \eqref{diffu1}, \eqref{add-4} and $\dfrac{\eps}{2}$
times \eqref{fifone}, we see that
\begin{align*}
  & \dfrac 12 \ddt\ZZZ+(1-\eps)\l  K(\vt_1)-K(\vt_2),\th-\tho\r
    +\|\nabla \bu\|^2+\dfrac \eps 2\|\fhi_t\|^2_{V'}\\
 & \quad \leq \eps \left(c+c\alpha^2+\delta\right) \|\th-\tho\|^2
   +3\alpha\eps \|\nabla \mu\|^2
   +3\beta \|\nabla\bu\|^2
   +g(t)\ZZZ,
\end{align*}
where $g$ was defined in \eqref{g}.

We now develop the second term in the \lhs. Owing to \eqref{defiK}, we actually have
\[
  \langle K(\vt_1) - K(\vt_2), \vt - \vt_{\Omega} \rangle
   = \|\vt - \vt_{\Omega}\|^2 + \frac{1}{q+1} \langle \ell(\vt_1) - \ell(\vt_2), \vt - \vt_{\Omega} \rangle,
\]
having set $\ell(\vt_i) = \vt_i^{q+1}$, $i = 1,2$. Now
\[
  \ell(\vt_1) - \ell(\vt_2)
    = \int_0^1 \dds \ell(s \vt_1 + (1-s) \vt_2) \, \dis
    = \int_0^1 \ell'(s \vt_1 + (1-s) \vt_2) (\vt_1 - \vt_2) \, \dis
    = \omega(\vt_1, \vt_2) \vt,
\]
where
\[
\omega(\vt_1, \vt_2) := \int_0^1 \ell'(s \vt_1 + (1-s) \vt_2) \, \dis
\]
and we notice that $\omega(\vt_1,\vt_2)\ge 0$ almost everywhere.
We also observe that, due to \eqref{rego:vt:step3},
\begin{equation}\label{stimadimu}
  |\omega(\vt_1, \vt_2)| \le c (1 + |\vt_1|^q + |\vt_2|^q) .
\end{equation}
Hence, it is not difficult to deduce
\begin{align*}
  \langle K(\vt_1) - K(\vt_2), \vt - \vt_{\Omega} \rangle
   & = \|\vt - \vt_{\Omega}\|^2 + \frac{1}{q+1} \langle \ell(\vt_1) - \ell(\vt_2), \vt - \vt_{\Omega} \rangle\\
  & \ge \|\vt - \vt_{\Omega}\|^2 + \frac{1}{q+1} \int_{\Omega} \omega(\vt_1, \vt_2)  \vt (\vt - \vt_{\Omega}) \\
  & = \|\vt - \vt_{\Omega}\|^2 + \frac{1}{q+1} \int_{\Omega} \omega(\vt_1, \vt_2) |\vt - \vt_{\Omega}|^2
   + \frac{1}{q+1} \int_{\Omega} \omega(\vt_1, \vt_2) \vt_{\Omega} (\vt - \vt_{\Omega})\\
 & \ge \|\vt - \vt_{\Omega}\|^2 + \frac{1}{q+1} \int_{\Omega} \omega(\vt_1, \vt_2) \vt_{\Omega} (\vt - \vt_{\Omega}).
\end{align*}
Therefore, using \eqref{rego:vt:step3} and \eqref{stimadimu}, we arrive at
\[
  \left |\frac{1}{q+1} \int_{\Omega} \omega(\vt_1, \vt_2) \vt_{\Omega} (\vt - \vt_{\Omega}) \right|
   \leq c \|\omega(\th_1,\th_2) \| | \tho | \|\th-\tho\|
   \leq \dfrac 12 \|\th-\tho\|^2 + c \tho^2,
\]
and, in turn,
\[
\langle K(\vt_1) - K(\vt_2), \vt - \vt_{\Omega} \rangle\ge \dfrac 12 \|\th-\tho\|^2 - c \tho^2.
\]

Summing up, we finally see that
\begin{align}\label{g21}
  & \dfrac 12 \ddt\ZZZ
    +\dfrac {1-\eps}2\|\th-\tho\|^2
    + \|\nabla \bu\|^2
    + \dfrac\eps 2\|\fhi_t\|^2_{V'}\\
 \nonumber
  & \quad \leq \eps \left(c+c\alpha^2+\delta\right)  \|\th-\tho\|^2
    + 3\alpha\eps \|\nabla \mu\|^2
    + 3\beta \|\nabla\bu\|^2
    + g(t) \ZZZ\\
 \nonumber
  & \quad \stackrel{\eqref{diffmu}}{\le} \eps \left(c+c\alpha^2+\delta\right)  \|\th-\tho\|^2
    + 3\alpha\eps c\|\fhi_t\|^2_{V'} 
    + 3\beta \|\nabla\bu\|^2
    + g(t)\ZZZ
    \end{align}
where in deducing the second inequality we used \eqref{g0},
the Friedrichs inequality \eqref{Friedr} together with \eqref{diffeffeV}, namely,
$$
   \| \nabla ( F'(\fhi_1) - F'(\fhi_2) ) \|
     \le c \| \nabla \fhi \|,
$$
as one can verify directly by using assumption \eqref{hp:F4}, the
previous estimates (cf.~\eqref{inizio}), and the fact that
$\fhi\OO = 0$. Moreover, the last inequality in \eqref{g21}
follows easily by comparing terms in \eqref{CH1d} and using
once more \eqref{inizio}. In particular, we can observe that
$$
   \| \nabla \mu \|
     \le c\big( \| \fhi_t \|_{V'} + \| \nabla \fhi \| + \| \vu \| \big).
$$
In particular, we may notice that the constant $c$ on the \rhs\ of
\eqref{g21} is {\sl independent}\/ of the parameters $\alpha,\beta,\delta,\eps$
(in fact it depends only on the regularity properties of solutions
collected in \eqref{inizio}).
As a consequence, we can first choose $\alpha,\beta,\delta>0$ small enough in order
to absorb the last two terms on the \rhs\ of \eqref{g21}
with the corresponding quantities appearing in the \lhs. In a second stage,
we also take $\eps$ sufficiently small (possibly depending
on the other parameters), in such a way that also the first
term on the \rhs\ is absorbed. Consequently, \eqref{g21}
eventually reduces to the simpler form
$$
  \ddt \ZZZ + \kappa_0 \left( \|\th-\tho\|^2+\|\nabla \bu\|^2+\|\fhi_t\|^2_{V'}\right) \leq g(t)\ZZZ,
$$
where $\kappa_0>0$ and $g$ has the same expression as in~\eqref{g} and therefore,
owing to \eqref{inizio}, it is summable over the interval $(0,T)$.
Being $\ZZZ(0) = 0$, then using Gronwall's
lemma we deduce that $\ZZZ$ is identically $0$ over $(0,T)$, whence the assert.


\section{Proof of Theorem \ref{teo:att}}
\label{quattro}

The proof is divided into various steps that are discussed in separate subsections
and in some occasion presented as single Lemmas. As before, $Q$ will denote a computable positive function,
increasingly monotone in each of its arguments,
whose expression is independent of time unless otherwise specified.
Hypotheses  \eqref{hp:F1}-\eqref{hp:F4} and \eqref{hp:kappa}
on the nonlinear terms will be always implicitly assumed in the sequel.


\subsection{Uniform estimates}

Theorem~\ref{teo:main} provides existence and uniqueness of  ``stable solutions'' on
fixed time intervals of arbitrary length. Our first step consists in showing that
the $\calV^r$-magnitude of any such solution remains bounded in a way only
depending on the initial data also as time goes to infinity.
\bele\label{dissintlemma}
 Any global solution $z(t)$ originating from an initial datum $z_0\in \calV^r$ satisfies
 \begin{equation}\label{u-bound}
  \|z(t)\|_{\mathcal H}+|\mu\OO(t)| \le
     Q(\|z_0\|_{\mathcal{H}})\quad \perogni t\geq 0.
 \end{equation}
 Moreover, we have the dissipation integrals
 \begin{equation} \label{diss-int}
   \int_{0}^{\infty} \io\left( |\nabla\vu(s)|^2+|\nabla \vt(s)|^2+|\nabla \mu(s)|^2\right) \, \dis
    + \sup_{t \ge 0} \int_{t}^{t + 1} \io|\nabla\Delta \fhi(s)|^2 \, \dis
      \le Q(\|z_0\|_{\mathcal{H}}).
 \end{equation}
\enle
\begin{proof}
Arguing as in \cite{ERS2}, we get the energy estimate, corresponding to the energy conservation principle,
testing \eqref{ns} by $\vu$, \eqref{CH1} by $\mu$, \eqref{CH2} by $\fhi_t$ and \eqref{calore} by $1$,
integrating over $\Omega$, and summing all the obtained relation together. Namely, we deduce
\[
  \ddt \mathcal{E}(\vu,\fhi,\teta) = 0,
\]
where
$$
  \mathcal E(\vu,\fhi,\teta)
    =\io \left( \dfrac12 |\vu|^2+\dfrac12 |\nabla\fhi|^2+F(\fhi)+\teta\right).
$$
This entails in particular that
\[
  \mathcal{E}(z(t)) = \mathcal{E}(z_0) \qquad \perogni t \ge 0.
\]
Therefore, using \eqref{czero}, we obtain
\begin{equation} \label{diss_stima1}
  \frac12 \|\vu(t)\|^2
   + \frac12 \|\nabla \varphi(t)\|^2
   + \int_{\Omega} \vt(t)
   - c_0 \le \mathcal{E}(z(t))
  = \mathcal{E}(z_0) \le Q(\|z_0\|_{\mathcal{H}}).
\end{equation}
Next, arguing once more as in \cite{ERS2}, we obtain the entropy estimate, corresponding to the entropy production principle.
It is obtained by testing \eqref{calore} by $- \vt^{-1}$ and integrating over $\Omega$, which yields
\begin{equation}\label{entropy_stima}
 \ddt \int_{\Omega} (- \log \vt)
  + \int_{\Omega} \frac{1}{\vt} \left(|\nabla \vu|^2 + |\nabla \mu|^2\right)
  + \int_{\Omega} \frac{\kappa(\vt)}{\vt^2} |\nabla \vt|^2
  = 0.
\end{equation}
Then, integrating in time over $(0,t)$ and recalling that, due to \eqref{cons_medie_energia},
\begin{equation}\label{medie:11}
  \varphi(t)_{\Omega} = (\varphi_0)_{\Omega},
\end{equation}
we deduce first of all that
\begin{equation}\label{controllo_log}
  - \int_{\Omega} \log \vt(t) \le - \int_{\Omega} \log \vt_0.
\end{equation}
Summing the above to \eqref{diss_stima1} and using that
$| \log r | \le r - \log r$ for all $r>0$, we then obtain
\begin{equation}\label{diss_stima2}
  \| \log \vt(t) \|_{L^1(\Omega)}
   \le Q(\|z_0\|_{\mathcal{H}}),
\end{equation}
for every $t \ge 0$.
%
%
Combining \eqref{diss_stima1} and \eqref{diss_stima2}, we then deduce
\begin{equation}\label{diss_stima3}
   \|z(t)\|_{\mathcal{H}}
    = \|\vu(t)\| + \|\fhi(t)\|_{V}
    + \|\vt(t)\|_{L^1(\Omega)}
    + \|\log \vt(t)\|_{L^1(\Omega)}
  \leq Q(\|z_0\|_{\mathcal{H}}),
\end{equation}
where the full $V$-norm of $\fhi(t)$ is controlled also in view of \eqref{medie:11}.
Finally, by \eqref{CH2}, we have
\[
  \mu(t)_{\Omega} = F'(\varphi(t))_{\Omega} - \vt(t)_{\Omega}.
\]
Hence, recalling also \eqref{hp:F4},
\[
  |\mu(t)_{\Omega}|
  \le Q(\|\varphi(t)\|_{V}, \vt(t)_{\Omega})
    \le Q(\|z_0\|_{\mathcal{H}}) \quad \perogni t \ge 0,
\]
and this result together with \eqref{diss_stima3} brings \eqref{u-bound}.

Let us now deal with \eqref{diss-int}. Integrating the entropy
estimate \eqref{entropy_stima} over $(0,t)$ for a generic $t>0$, we deduce
\begin{align} \label{int-ent-est}
  0 & \le \int_{0}^{t} \int_{\Omega} \frac{1}{\vt} \left(|\nabla \vu|^2 + |\nabla \mu|^2\right)
   + \int_{0}^{t} \int_{\Omega} \frac{\kappa(\vt)}{\vt^2} |\nabla \vt|^2 \\
 \nonumber
  & = \int_{\Omega} \log \vt(t)
    - \int_{\Omega} \log \vt_0
   \stackrel{\eqref{u-bound}}{\le} Q(\|z_0\|_{\mathcal{H}}).
\end{align}
Arguing as in \cite{ERS2}, we then deduce
\begin{align}\label{nablavtheta}
  \int_{0}^{t} \int_{\Omega} |\nabla \vt|^2
   & \le c \int_{0}^{t} \int_{\Omega} \left (\frac{1}{\vt^2}
    + k_q \vt^{q-2} \right ) |\nabla \vt|^2 \\
 \nonumber
  & \le c \int_{0}^{t} \int_{\Omega} \frac{\kappa(\vt)}{\vt^2} |\nabla \vt|^2
    \stackrel{\eqref{int-ent-est}}{\le} Q(\|z_0\|_{\mathcal{H}})
      \quad \perogni t \ge 0,
\end{align}
where we remark that the quantity on the \rhs\ is independent of $t$.
Then, integrating \eqref{calore} over~$\Omega$ and using the periodic boundary
conditions, we infer
\[
  \|\nabla \vu\|^2 + \|\nabla \mu\|^2
   = \ddt \int_{\Omega} \vt + \int_{\Omega} \vt \Delta \mu
   \le \ddt \int_{\Omega} \vt + \frac12 \| \nabla \vt \|^2 + \frac12 \| \nabla \mu \|^2.
\]
Hence, integrating in time over $(0,t)$, absorbing the last term on the
\rhs\ with the corresponding one on the \lhs, and recalling \eqref{nablavtheta}
and \eqref{diss_stima3}, we readily arrive at
%
%
%
%
%
%
\begin{equation}\label{nablavumu}
  \int_0^t \left ( \|\nabla \vu\|^2 + \|\nabla \mu\|^2 \right )
    \le  Q(\|z_0\|_{\calH}), \quad \perogni t\ge 0.
\end{equation}
Collecting \eqref{nablavtheta} and \eqref{nablavumu}, observing once more that the quantities on the
\rhs\ are independent of $t$, and letting $t\nearrow\infty$,
we then deduce the first bound in \eqref{diss-int}. Next,
testing \eqref{CH2} by $\Delta^2\fhi$ and performing
standard manipulations (see \cite[(3.20)]{ERS2} for details), we deduce
\[
  \|\nabla \Delta \varphi\|^2
   \le c \left ( Q(\|\fhi\|_{V})
       + \|\nabla \vt\|^2
       + \|\nabla \mu\|^2 \right)
   \stackrel{\eqref{diss_stima3}}{\le} c \left(Q(\|z_0\|_{\mathcal{H}})
       + \|\nabla \vt\|^2
       + \|\nabla \mu\|^2 \right).
\]
Then, integrating the above relation over a generic interval $(t,t+1)$ and
using \eqref{nablavtheta} and \eqref{nablavumu}, we eventually obtain the
second bound in \eqref{diss-int}.
\end{proof}
\noindent%
It is worth noting that the above lemma is stated for ``stable solutions'' (i.e.,
for initial data $z_0\in \calV^r$); however, only the $\calH$-regularity
of the initial datum is used in the proof, and consequently the assert
remains valid for more general classes of solutions. The same holds for the next
result, which is an immediate consequence of the lemma.
\beco\label{coro:y}
 Under the same assumptions as in\/ {\rm Lemma~\ref{dissintlemma}}, for any $z_0\in \calV^r$
 and any $t\ge 0$, there exists $s\in (t,t+1)$ such that
 \beeq{ubn}
   \|z(s)\|_{\mathcal{V}} + \|\mu(s)\|_V
    = \|\vu(s)\|_{V}
     + \|\fhi(s)\|_{H^3(\Omega)}
     + \|\vt(s)\|_V
     + \|\mu(s)\|_{V}
   \leq Q(\|z_0\|_{\calH}).
 \eddeq
\enco
\noindent%
Actually, a similar estimate holds also for the ``$\calV^r$-magnitude''
of the solution (of course now assuming the initial datum to lie in $\calV^r$
is essential). This is the object of the next
\bele\label{lemma:x}
 Given $z_0\in \mathcal{V}^r$, for any $t\geq 0$ there exists $s\in [t,t+1)$ such that
 \begin{align}\label{ubn2}
   \|z(s)\|_{\mathcal V^r}
     + \|\mu(s)\|_V
    & = \|\vu(s)\|_{H^{r+1}(\Omega)}
     + \|\fhi(s)\|_{H^3(\Omega)}
     + \|\vt(s)\|_V
     + \|(1/\teta(s))\|_{L^1(\Omega)}
     + \|\mu(s)\|_V \\
  \nonumber
    & \leq Q(\|z_0\|_{\mathcal{V}^r}).
 \end{align}
\enle
\begin{proof}
Being $z_0\in\calV^r$, we have in particular
$(1/\vt_0)\in L^1(\Omega)$. Hence we may apply
Lemma~\ref{thetaalpha} with $\alpha=2$ over the time interval
$(0,+\infty)$. Note, indeed, that the quantity noted as $N$
in \eqref{nablate} is globally controlled due to~\eqref{diss-int}.
Thus, we deduce
\begin{equation}\label{sempret}
  \|\teta^{-1}(t)\|_{L^1(\Omega)}
   \leq Q(\|z_0\|_{\mathcal V}),\quad \perogni t\geq 0.
\end{equation}
Hence, it only remains to improve the regularity estimate on $\vu$.
To this aim, let us take any $t\ge 0$ and observe that, thanks to
\eqref{diss-int}, there exists $\tau\in (t,t+1/2)$ such that
$\| \vu(\tau) \|_V \le Q(\|z_0\|_{\mathcal H})$. Then, test \eqref{ns}
by $-\Delta \vu$, and integrate over $(\tau,\tau+1/2)\times \Omega$.
Performing standard manipulations and using the regularity of $\vu(\tau)$
and the uniform bounds \eqref{u-bound} and \eqref{diss-int},
it is not difficult to deduce
$$
  \int_\tau^{\tau+1/2} \| \Delta \vu \|^2
   \leq Q(\|z_0\|_{\mathcal H}).
$$
As a consequence, there exists $s\in (\tau,\tau+1/2)$
(so that in particular $s\in (t,t+1)$) such that
$\| \vu(s) \|_{H^2(\Omega)} \le Q(\|z_0\|_{\mathcal H})$.
Hence we have in particular \eqref{ubn2}.
\end{proof}
\noindent
The next property plays a key role in the asymptotic
analysis of stable solutions. Namely, we can prove that
if the initial datum $z_0$ lies in $\calV^r$ then the
$\calV^r$-magnitude of $z(t)$ is controlled uniformly in time:
\bele
 Given $z_0\in \mathcal{V}^r$, we have the bound
 \begin{equation}\label{lemma4.3}
   \|z(t)\|_{\mathcal{V}^r}
     \leq Q(\|z_0\|_{\mathcal{V}^r}), \quad \text{for a.a.~} t\in(0,\infty).
 \end{equation}
 Moreover, for any $s\ge 0$, there holds the following
 additional bound:
 \begin{align}\label{dal:main3}
   & \| \mu \|_{L^2(s,s+2;H^3(\Omega))}
    + \| \vu_t \|_{L^2(s,s+2;H^{r}(\Omega))}
    + \| \vu \|_{L^2(s,s+2;H^{2+r}(\Omega))}
    + \| \fhi_t \|_{L^2(s,s+2;V)} \\
 \nonumber
  & \qquad\quad
    + \| \vt \|_{L^2(s,s+2;H^2(\Omega))}
    + \| \vt_t \|_{L^2(s,s+2;H)}
     \le Q(\| z_0 \|_{\calV^r}),
 \end{align}
 with $Q$ independent of the choice of $s$.
\enle
\begin{proof}
Given any $s\ge 0$ such that $z(s)\in \calV^r$,
we may interpret $z(s)$ as an ``initial'' datum and
apply Theorem~\ref{teo:main} over the time interval
$(s,s+2)$ (in place of $(0,T)$). Then, the solution
satisfies the regularity properties
\eqref{rego:vu:new}-\eqref{rego:1vt:new}
over $(s,s+2)$. More precisely, one has the
quantitative estimates
\begin{align}\label{dal:main2}
  & \| \vu \|_{L^\infty(s,s+2;H^{1+r}(\Omega))}
   + \| \fhi \|_{L^\infty(s,s+2;H^3(\Omega))}
   + \| \mu \|_{L^\infty(s,s+2;V)}\\
 \nonumber
  & \quad\quad
   + \| \vt \|_{L^\infty(s,s+2;V)}
   + \| K(\vt) \|_{L^\infty(s,s+2;V)}
   + \| 1/\vt \|_{L^\infty(s,s+2;L^{1}(\Omega))}
  \le Q(\| z(s) \|_{\calV^r})
\end{align}
and
\begin{align}\label{dal:main3x}
  & \| \mu \|_{L^2(s,s+2;H^3(\Omega))}
   + \| \vu_t \|_{L^2(s,s+2;H^{r}(\Omega))}
   + \| \vu \|_{L^2(s,s+2;H^{2+r}(\Omega))}
   + \| \fhi_t \|_{L^2(s,s+2;V)} \\
 \nonumber
 & \qquad\quad
   + \| \vt \|_{L^2(s,s+2;H^2(\Omega))}
   + \| \vt_t \|_{L^2(s,s+2;H)}
    \le Q(\| z(s) \|_{\calV^r}),
\end{align}
where in this case $Q$ may depend on the length of the considered
time span, which is however fixed (and equal to 2).
Hence, taking first $s=0$, which is possible
because $z_0\in \calV^r$ by assumption
and then with the choices of $s$ provided by
Lemma~\ref{lemma:x}, we readily get the assert
(clearly \eqref{dal:main3x} reduces to \eqref{dal:main3}
and \eqref{dal:main2}, written concisely, becomes
\eqref{lemma4.3}). Indeed, we may notice that the $L^1$-norm of $1/\vt$
(which is a summand in the quantity $\| z \|_{\calV^r}$)
has already been controlled uniformly over $(0,\infty)$ by virtue of
\eqref{sempret}.
\end{proof}
%


\subsection{Asymptotic compactness}

We are now ready to show asymptotic compactness of trajectories
associated to stable solutions. Namely, if the initial datum $z_0$ lies
in $\calV^r$, then for any $t>0$ the corresponding solutions
$z(t)$ belongs to  $\calW$. Moreover, the $\calW$-magnitude of
$z(t)$ is uniformly bounded for large $t$. This is stated
more precisely in the
\bele\label{ascpt-le}
 Let $z_0\in \mathcal V^r$. Then, for any $\tau>0$ there holds
 \begin{equation}\label{ACtotal}
   \|z\|_{L^\infty(\tau,+\infty;\mathcal W)}\leq Q( \|z_0\|_{\calV^r}, \tau^{-1}).
 \end{equation}
\enle
\begin{proof}
We need to prove the following regularity properties (and  control  the corresponding norms  by
a quantity $Q(\|z_0\|_{\calV^r},\tau^{-1})$):
\begin{itemize}
\item[(vii)] $1/\teta\in L^\infty(\tau,+\infty;L^4(\Omega))$ and $\nabla(1/\teta)\in L^\infty(\tau,+\infty;L^1(\Omega))$;
\item[(viii)] $\vu\in L^\infty(\tau,+\infty; H^2(\Omega))$
and $\vu\in L^2(t, t + 1, H^3(\Omega))$ for all $t > \tau$;
\item[(ix)] $\teta \in L^{\infty}(\tau,+\infty; H^2(\Omega))$;
\item[(x)] $\fhi\in L^\infty(\tau,+\infty;H^4(\Omega))$.
\end{itemize}
\noindent
$\bullet \,$ \textsc{proof of \textnormal{(vii)}}.
Let us go back to formula \eqref{alfetta1} in Lemma \ref{thetaalpha}
with $\alpha=2$. This contains the additional estimate
\begin{equation}\label{asym:11}
  \| \nabla \teta^{-1/2}\|_{L^2(0,+\infty;H)}
   \le Q(\|z_0\|_{\calV^r}).
\end{equation}
Combining this property with the information on $1/\vt$ provided by~\eqref{lemma4.3}
and using the continuity of the embedding $V\subset L^p(\Omega)$ for any $p\in[1,\infty)$, we
deduce
\begin{equation}\label{asym:12}
  \sup_{t\ge 0} \| \teta^{-1}\|_{L^1(t,t+1;L^p(\Omega))}
   \le Q(\|z_0\|_{\calV^r},p).
\end{equation}
Let now apply once more Lemma \ref{thetaalpha} for a generic $\alpha\ge 2$.
Then, the differential version of \eqref{alfetta1} reads
\begin{equation} \label{tetaprimo}
  \ddt \io \teta^{1-\alpha} + c_\alpha \| \nabla \vt^{(1-\alpha)/2} \|^2
   \leq c_\alpha' \|\nabla\teta\|^2,
\end{equation}
for some $c_\alpha,c_\alpha'>0$.
For our purposes it is enough to take $\alpha=5$. Then,
from \eqref{asym:12} there follows that, for any $t\ge 0$
and $\tau>0$, there exists $s\in [t,t+\tau]$ such that
\begin{equation} \label{tetaprimo2}
  \| \teta^{-1}(s) \|_{L^4(\Omega)}
   \le \frac1\tau \int_t^{t+\tau} \| \teta^{-1}(r) \|_{L^4(\Omega)} \,\dir
   \le \frac1\tau  Q(\|z_0\|_{\calV^r})
   = Q(\|z_0\|_{\calV^r},\tau^{-1}).
\end{equation}
Hence, choosing $t=0$ and $s\in(0,\tau)$ such that \eqref{tetaprimo2} holds,
integrating \eqref{tetaprimo} (with $\alpha=5$)
over $(s,+\infty)$, and recalling \eqref{diss-int}, we readily obtain
the first of~(vii). 
%
%
Then, recalling \eqref{dal:main2}, for any $t\ge \tau>0$, we have
$$
  \| \nabla \teta^{-1}(t) \|_{L^1(\Omega)}
   \le \| \nabla \teta(t) \| \, \| \teta^{-1}(t) \|_{L^4(\Omega)}^2
   \le Q(\|z_0\|_{\calV^r},\tau^{-1}),
$$
whence follows the second of~(vii).
\beos\label{rem:moser}
 An easy refinement of the argument above yields more precisely
  $$
    \| \teta^{-1} \|_{L^\infty(\tau,\infty;L^p(\Omega))}
     \le Q(\|z_0\|_{\calV^r},\tau^{-1},p)
  $$
  for all $\tau>0$ and all $p\in[1,\infty)$. Actually,
  at the price of some additional work, in the
  spirit of Moser's iterations one may also prove that
 \begin{equation}  \label{iter:11c}
   \| \vt^{-1} \|_{L^\infty(\tau,\infty;L^\infty(\Omega))}
    \le Q(\|z_0\|_{\calV^r},\tau^{-1}).
 \end{equation}
 We omit details because we do not want to overburden the reader with
 technical arguments.
\eddos

\smallskip

\noindent%
$\bullet$ \, \textsc{proof of \textnormal{(viii)}}.
In the sequel we shall denote by $C$ a generic positive constant of the form
$C=Q(\|z_0\|_{\calV^r},\tau^{-1})$. Namely, $C$ may first of all depend on the
quantities uniformly estimated by \eqref{lemma4.3}. Moreover, it may also depend
on time in the sense that it may explode (in a controlled and computable
way) as $\tau\searrow 0$. This is a natural behavior as we are looking for
parabolic smoothing estimates.

That said, to prove the additional regularity for $\vu$
we proceed as in the proof of (i) in Subsection~\ref{part1}:
we first project equation \eqref{ns} into the space $\mathbb{H}$ by applying the operator $\mathbb{P}$,
then we test by $A^2 \hat\vu$, where $\hat{\vu} = \vu - \vu_{\Omega}$.
Recalling \eqref{projf}, we obtain
\begin{align}\label{eq:Ar:bis}
  \frac{1}{2} \ddt \|A \hat \vu\|^2 + \|A^{3/2} \hat \vu\|^2
   & = - b (\hat \vu, \hat \vu, A^{2} \hat \vu) + \langle f, A^{2} \hat \vu \rangle \\
 \nonumber
  & \le - b (\hat \vu, \hat \vu, A^{2} \hat \vu)
       + \frac{1}{4} \|A^{3/2} \hat \vu\|^2 + C \left (\|\varphi\|^3_{H^3(\Omega)} + 1 \right ),
\end{align}
where we also used \eqref{rego:Temam:2}. On the other hand,
\begin{align*}
   |\langle (\hat \vu \cdot \nabla) \hat \vu, A^2 \hat \vu \rangle |
  & = |\langle A^{1/2}(\hat \vu \cdot \nabla \hat \vu), A^{3/2} \hat \vu \rangle | \\
  & \le \int_{\Omega} (|\nabla \hat \vu|^2 + |\hat \vu| |\nabla^2 \hat \vu|) \, |\nabla^3 \hat \vu| \\
  & \le \|\nabla \hat \bu\|^2_{L^4(\Omega)} \, \|\nabla^3 \hat \bu\| + \|\hat \bu\|_{L^{\infty}(\Omega)} \|\nabla^2 \hat \bu\| \, \|\nabla^3 \hat \bu\| \\
  & \stackrel{\eqref{dis:L4}, \eqref{dis:Linfty}}{\le} C \|\nabla \hat \bu\| \, \|\nabla^2 \hat \bu\| \, \|\nabla^3 \hat \bu\|
    + C \|\nabla^2 \hat \bu\|^{3/2} \, \|\hat \bu\|^{1/2} \, \|\nabla^3 \hat \bu\| \\
  & \stackrel{\eqref{dis:BrezziGilardi}}{\le} \frac{1}{8} \|\nabla^3 \hat \bu\|^2 + C \|\nabla^2 \hat\bu\|^2 + C \|\nabla^3 \hat \bu\|^{7/4}\\
  &\le \frac{1}{4} \|\nabla^3 \hat \bu\|^2 + C \big( \|\nabla^2 \hat\bu\|^2 + 1 \big).
\end{align*}
Coming back to \eqref{eq:Ar:bis} we then obtain
\[
  \frac{1}{2} \ddt \|A \hat \vu\|^2 + \frac{1}{2} \|A^{3/2} \hat \vu\|^2
   \le C \left (\|\varphi\|_{H^3(\Omega)}^3 + 1 \right ) + C \|\nabla^2 \hat \vu\|^2.
\]
Then, the thesis is obtained by Gronwall's lemma by also exploiting \eqref{dal:main2}.
To be more precise, we integrate the above relation over time
intervals of fixed length (for instance equal to~$2$) taking
as starting point suitable times $s$ such that $\|\vu(s)\|_{H^2(\Omega)}$ is controlled
by the $\calV^r$-magnitude of the initial datum (these times $s$ are characterized in
the proof of Lemma~\ref{lemma:x}). This yields (viii) away from $0$.
Then, in order to control the boundary layer at $t\sim 0$, we multiply the above inequality
by $t\in(0,1)$ obtaining
\[
  \ddt \left( \frac{t}{2} \|A \hat \vu\|^2 \right) + \frac{t}{2} \|A^{3/2} \hat \vu\|^2
   \le \frac12 \|A \hat \vu\|^2 + C \left (\|\varphi\|_{H^3(\Omega)}^3 + 1 \right ) + C \|\nabla^2 \hat \vu\|^2.
\]
Integrating over $(0,1)$ and using \eqref{dal:main3}, we get~(viii)
for $t\sim 0$, as desired.

\smallskip

\noindent%
$\bullet \,$ \textsc{proof of \textnormal{(ix)}}.
We test \eqref{calore} by $\Delta^2 \vt$ to obtain
\begin{equation}\label{calore-testN}
   \frac{1}{2} \ddt \|\Delta \vt\|^2 + \langle \nabla \Delta K(\vt), \nabla \Delta \vt \rangle
     + \langle \bu \cdot \nabla \vt, \Delta^2 \vt \rangle
     + \langle \vt \Delta \mu, \Delta^2 \vt \rangle
    = \langle |\nabla \bu|^2 + |\nabla \mu|^2, \Delta^2 \vt \rangle.
\end{equation}
At this point, we notice that for a generic function $\phi$
\[
  \|\phi\|^2
    = \|\phi - \phi_{\Omega} + \phi_{\Omega}\|^2
     = \|\phi - \phi_{\Omega}\|^2 +  |\Omega| \phi_{\Omega}^2
     \stackrel{\eqref{Friedr}}{\le}  c_{\Omega} \|\nabla \phi\|^2 +  |\Omega| \phi_{\Omega}^2.
\]
On the other hand,
\[
  \|\nabla \phi\|^2
    = \langle \phi - \phi_{\Omega}, - \Delta \phi \rangle
    \le \|\phi - \phi_{\Omega}\| \, \|\Delta \phi\|
    \le \frac{1}{2 c_{\Omega}} \|\phi - \phi_{\Omega}\|^2
       + c \|\Delta \phi\|^2
    \stackrel{\eqref{Friedr}}{\le} \frac{1}{2} \|\nabla \phi\|^2 + c \, \|\Delta \phi\|^2.
\]
Moreover, we will use repeatedly in the sequel the following facts:
\begin{equation}\label{norma-equiv}
  \|\Delta \phi\|^2
   \le \|\phi\|^2_{H^2(\Omega)}
    \le |\Omega| \phi^2_{\Omega} + c \|\Delta \phi\|^2.
\end{equation}
We analyze all terms in \eqref{calore-testN}. We have
\[
  \Delta K(\vt) = \dive (\kappa(\vt) \nabla \vt) = \kappa'(\vt) |\nabla \vt|^2 + \kappa(\vt) \Delta \vt.
\]
Therefore,
\[
\nabla \Delta K(\vt) = \kappa''(\vt)  |\nabla \vt|^2 \nabla \vt 
+ 2 \kappa'(\vt) \nabla \vt \cdot \nabla^2 \vt 
+ \kappa'(\vt) \nabla \vt \Delta \vt + \kappa(\vt) \nabla \Delta \vt.
\]
This permits us to deal with the second term in \eqref{calore-testN}:
\begin{align*}
  \langle \nabla \Delta K(\vt), \nabla \Delta \vt \rangle
   & = \int_{\Omega} \kappa(\vt) |\nabla \Delta \vt|^2
     + \langle  \kappa''(\vt)  |\nabla \vt|^2 \nabla \vt
     + 2 \kappa'(\vt) \nabla \vt \cdot \nabla^2 \vt
     + \kappa'(\vt) \nabla \vt \Delta \vt, \nabla \Delta \vt \rangle\\
   & =: \int_{\Omega} \big(\sqrt{\kappa(\vt)} |\nabla \Delta \vt| \big)^2 + \Theta_1 + \Theta_2 + \Theta_3.
\end{align*}
To estimate the \rhs\ let us assume, for simplicity, that $q\ge 4$ (the opposite situation
being in fact easier). Then we first have
\begin{align*}
  |\Theta_1| & \le q(q-1) \int_{\Omega} \vt^{q-2} |\nabla \Delta \vt| |\nabla \vt|^3
   = q(q-1) \int_{\Omega} \vt^{q/2} |\nabla \Delta \vt| (\vt^{\frac{q}{2} - 2}  |\nabla \vt|^{\frac{1}{2}}) |\nabla \vt|^{\frac{5}{2}} \\
  & \le C \|\sqrt{\kappa(\vt)} |\nabla \Delta \vt|\| \, \| | \nabla K(\vt)|^{\frac{1}{2}} \|_{L^{4}(\Omega)}  \, \| |\nabla \vt|^{\frac{5}{2}} \|_{L^{4}(\Omega)}\\
  & \le C \|\sqrt{\kappa(\vt)} |\nabla \Delta \vt|\| \, \|\nabla K(\vt)\|^{\frac{1}{2}} \, \|\nabla \vt\|^{\frac{5}{2}}_{L^{10}(\Omega)}
   \stackrel{\eqref{dal:main2}, \eqref{dis:Brezis}}{\le} C \|\sqrt{\kappa(\vt)} |\nabla \Delta \vt|\|
      \left( \|\nabla \vt\|^{\frac{1}{5}} \, \|\vt\|^{\frac{4}{5}}_{H^2(\Omega)} \right)^{\frac{5}{2}} \\
 & \le C \|\sqrt{\kappa(\vt)} |\nabla \Delta \vt|\| \,\|\vt\|^2_{H^2(\Omega)}
   \le \varepsilon \|\sqrt{\kappa(\vt)} |\nabla \Delta \vt|\|^2 + C_{\varepsilon} \|\vt\|^4_{H^2(\Omega)}.
\end{align*}
On the other hand,
\begin{align*}
  |\Theta_2| & \le 2\int_{\Omega} q \vt^{q-1} |\nabla \vt| \, |\nabla^2 \vt| \, |\nabla \Delta \vt| \\
  & \le C \| \vt^{q/2} |\nabla \Delta \vt| \| \, \| \vt^{(q-2)/2} | \nabla \vt|^{1/2} \|_{L^4(\Omega)}
   \, \| | \nabla \vt|^{1/2} \|_{L^8(\Omega)} \, \| \nabla^2 \vt \|_{L^8(\Omega)}\\
  &\stackrel{\eqref{dis:Linfty}, \eqref{dis:L4}}{\le} C \|\sqrt{\kappa(\vt)} |\nabla \Delta \vt| \| \, \|\nabla K(\vt)\|^{\frac{1}{2}} \|\nabla \vt\|^{\frac{1}{4}}
        \, \|\vt\|_{H^2(\Omega)}^{\frac{1}{2}} \, \|\vt\|^{\frac{3}{4}}_{H^3(\Omega)}\\
  & \stackrel{\eqref{rego:Kvt:new}}{\le} C \|\sqrt{\kappa(\vt)} |\nabla \Delta \vt| \|^{\frac{7}{4}} \|\vt\|^{\frac{1}{2}}_{H^2(\Omega)}
   \le \varepsilon \|\sqrt{\kappa(\vt)} |\nabla \Delta \vt| \|^2
     + C_\varepsilon \|\vt\|_{H^2(\Omega)}^4
\end{align*}
and it is apparent that the term $\Theta_3$ can be controlled analogously.

We now deal with the third term in \eqref{calore-testN}. Due to the fact that
\begin{equation}\label{wellknown}
  \nabla (\bu \cdot \nabla \vt) = \nabla \vt  \cdot \nabla \bu + (\bu \cdot \nabla) \nabla \vt,
\end{equation}
we then have
\begin{align*}
  \langle \bu \cdot \nabla \vt, \Delta^2 \vt \rangle
    & = - \langle \nabla (\bu \cdot \nabla \vt), \nabla \Delta \vt \rangle
      \stackrel{\eqref{wellknown}}{=} - \langle \nabla \vt  \cdot \nabla \bu, \nabla \Delta \vt \rangle + \langle (\bu \cdot \nabla) \nabla \vt, \nabla \Delta \vt \rangle\\
  & \le 2 \varepsilon \|\nabla \Delta \vt\|^2 + C_{\varepsilon}\|\nabla \vt\|^2_{L^4(\Omega)} \|\nabla \bu\|^2_{L^4(\Omega)} + C_{\varepsilon} \int_{\Omega} |\bu|^2 \, |\nabla^2 \vt|^2 \\
  & \le 2 \varepsilon \|\nabla \Delta \vt\|^2 + C_{\varepsilon} \|\vt\|_{V} \|\vt\|_{H^2(\Omega)} \|\bu\|_{V} \|\bu\|_{H^2(\Omega)}
      \stackrel{\eqref{dis:Brezis}}{+} C_{\varepsilon}  \|\bu\|^2_{L^6(\Omega)} \|\nabla^2 \vt\|^{4/3} \|\nabla^3 \vt\|^{2/3}\\
  & \le 3 \varepsilon \|\sqrt{\kappa(\vt)} \nabla \Delta \vt\|^2 \stackrel{\eqref{dis:Brezis}}{+} C_{\varepsilon} \|\nabla^2 \vt\|^{2} + C_\varepsilon,
\end{align*}
where we used (viii) of Lemma~\ref{ascpt-le} together with (i) and (ii) of Subsection~\ref{part1}.
Notice in particular that at this level the constants $C$ (and $C_\varepsilon$)
are allowed to depend also on the $H^2$-norm of $\vu$,
which has been estimated in~(viii) far from $0$. Hence, here (and below) $C$ is of the form
$C=C(t)=Q(\| z_0 \|_{\calV^r},t^{-1})$.

Next, concerning the fourth term in \eqref{calore-testN} we have
\[
  |\langle \vt \Delta \mu, \Delta^2 \vt \rangle|
    \le \int_{\Omega} |\Delta \mu| \, |\nabla \vt|\,|\nabla \Delta \vt|
      + \int_{\Omega} |\vt| \, |\nabla \Delta \mu| \, |\nabla \Delta \vt|
      =: \Theta_4 + \Theta_5,
\]
where
\begin{align*}
  \Theta_4 & := \int_{\Omega} |\Delta \mu| \, |\nabla \vt|\,|\nabla \Delta \vt|
    \le \|\nabla \Delta \vt\| \, \|\Delta \mu\|_{L^6(\Omega)} \, \|\nabla \vt\|_{L^3(\Omega)}\\
  & \le C \|\nabla \Delta \vt\| \,\|\Delta \mu\|^{\frac{1}{3}} \, \|\nabla \Delta \mu\|^{\frac{2}{3}}  \, \|\nabla \vt\|^{\frac{2}{3}} \, \|\vt\|_{H^2(\Omega)}^{\frac{1}{3}}\\
  & \stackrel{\eqref{dis:BrezziGilardi}}{\le} C  \|\nabla \Delta \vt\| \, \|\nabla \mu\|^{\frac{1}{6}} \, \|\nabla \Delta \mu\|^{\frac{2}{3}
    + \frac{1}{6}} \|\nabla \vt\|^{\frac{2}{3}} \, \|\vt\|_{H^2(\Omega)}^{\frac{1}{3}}\\
  & \le C \|\nabla \Delta \vt\| \, \|\nabla \Delta \mu\|^{\frac{5}{6}} \, \|\vt\|^{\frac{1}{3}}_{H^2(\Omega)}\\
  & \le \varepsilon \|\sqrt{\kappa(\vt)} \nabla \Delta \vt\|^2 + C_\varepsilon \|\nabla \Delta \mu\|^2  +C_\varepsilon \|\vt\|^4_{H^2(\Omega)},
\end{align*}
where we also used the uniform control \eqref{dal:main2} of $\mu$.
On the other hand,
\begin{align*}
  \Theta_5 & := \int_{\Omega} |\vt| \, |\nabla \Delta \mu| \, |\nabla \Delta \vt| \\
  & \le \|\sqrt{\kappa(\vt)}  |\nabla \Delta \vt| \| \, \|\nabla \Delta \mu\| \\
  & \le \varepsilon \|\sqrt{\kappa(\vt)} |\nabla \Delta \vt| \|^2
    +  C_{\varepsilon}\|\nabla \Delta \mu\|^2.
\end{align*}
Coming to the very last terms in \eqref{calore-testN}, we eventually have
\begin{align*}
  |\langle |\nabla \bu|^2, \Delta^2 \vt \rangle |
    & \le C \int_{\Omega} |\nabla^2 \bu \, \nabla \bu \, \nabla \Delta \vt|
      \le C \|\nabla^2 \bu\|_{L^4(\Omega)} \, \|\nabla \bu\|_{L^4(\Omega)} \, \|\nabla \Delta \vt\| \\
    & \le C \|\nabla^2 \bu\|^{1/2}_{V} \, \|\nabla^2 \bu\| \, \|\nabla \bu\|^{1/2} \, \|\nabla \Delta \vt\|
     \le \varepsilon \|\sqrt{\kappa(\vt)} \nabla \Delta \vt\|^2
      +  C_{\varepsilon}  \|\nabla^3 \bu\|^2,
\end{align*}
where we used \eqref{lemma4.3} and the previous estimate (viii). Similarly, we also get
\begin{align*}
  |\langle |\nabla \mu|^2, \Delta^2 \vt \rangle |
    & \le C \int_{\Omega} | \nabla^2 \mu \, \nabla \mu \, \nabla \Delta \vt|
      \le C \|\nabla^2 \mu\| \, \|\nabla \mu\|_{L^{\infty}(\Omega)} \, \|\nabla \Delta \vt\| \\
  & \le C \| \mu\|_{H^2} \| \mu\|^{1/2}_{H^3}\| \mu\|^{1/2}_{H^1} \, \|\nabla \Delta \vt\|
    \le C \| \mu\|_{H^3}\| \mu\|_{H^1} \, \|\nabla \Delta \vt\|\\
 & \le \varepsilon \|\sqrt{\kappa(\vt)} \nabla \Delta \vt\|^2
        +  C_{\varepsilon}  \|\mu\|^2_{H^3}.
\end{align*}
Collecting the above estimates, and taking $\varepsilon$ small enough,
we finally arrive at
\begin{equation}\label{asym:31}
  \ddt \|\Delta \vt\|^2
   + \|\sqrt{\kappa(\vt)} |\nabla\Delta\vt|\|^2
  \leq C \|\vt\|^2_{H^2} \, \|\Delta\vt\|^2 + C ( 1 + \|\mu\|^2_{H^3} + \|\bu\|^2_{H^3} + \|\vt\|^2_{H^2} ),
\end{equation}
where, as said, $C=C(t)=Q(\| z_0 \|_{\calV^r},t^{-1})$. Then, similarly as before,
we apply Gronwall's lemma on time intervals of the form $(s,s+2)$ starting from
suitable $s$ such that $\|\Delta \vt(s)\|^2$ is controlled.
Indeed, from \eqref{dal:main3} one can easily deduce that for any $t\ge 0$ and
any $\tau>0$ there exists $s\in [t,t+\tau]$ such that
$\| \vt(s) \|_{H^2(\Omega)} \le Q(\| z_0 \|_{\calV^r},\tau^{-1})$.
Then, noting that the \rhs\ of \eqref{asym:31}
is controlled by means of \eqref{dal:main3} and~(viii), we deduce~(ix).

\smallskip

\noindent%
$\bullet$ \, \textsc{proof of \textnormal{(x)}}.
To derive the  additional regularity of $\varphi$, we first prove that
\begin{equation}\label{stimavarphit}
  \| \varphi_t \|_{L^{\infty}(\tau, +\infty; H)}
   \le Q(\| z_0 \|_{\calV^r},\tau^{-1}).
\end{equation}
To this aim, we differentiate in time \eqref{CH1}-\eqref{CH2}, obtaining
\begin{align}
  & \varphi_{tt} + \vu_t \cdot \nabla \varphi + \vu \cdot \nabla \varphi_t = \Delta \mu_t, \label{CH1bis}\\
  & \mu_t = - \Delta \varphi_t + F''(\varphi) \varphi_t - \vt_t. \label{CH2bis}
\end{align}
Then we test \eqref{CH1bis} by $\varphi_t$ and \eqref{CH2bis} by  $\Delta \varphi_t$  and sum the resulting
relations in order to erase the term $\langle \Delta \mu_t, \varphi_t \rangle$. Moreover, noticing that
\[
  \langle \vu \cdot \nabla \varphi_t, \varphi_t \rangle = 0
\]
due to \eqref{incom}, we obtain
\[
  \frac{1}{2} \ddt \|\varphi_t\|^2 + \|\Delta \varphi_t\|^2
   = - \langle \vu_t \cdot \nabla \varphi, \varphi_t \rangle - \langle \vt_t, \Delta \varphi_t \rangle
     + \langle F''(\varphi) \varphi_t, \Delta \varphi_t \rangle.
\]
Let us now notice that
\begin{align*}
  - \int_{\Omega} \vu_t \cdot \nabla \varphi \, \varphi_t
    & \le \|\vu_t\|\, \|\nabla \varphi\| \, \|\varphi_t\|_{L^{\infty}(\Omega)}
     \stackrel{\eqref{dis:Linfty}}{\le} C \|\vu_t\| \, \|\nabla \varphi\| \, \|\varphi_t\|^{1/2} \, \|\Delta \varphi_t\|^{1/2} \\
   & \le \frac{1}{6} \|\Delta \varphi_t\|^2 + C \, \|\varphi_t\|^2 + C \, \|\vu_t\|^2,
\end{align*}
thanks also to \eqref{lemma4.3}. On the other hand,
\[
  - \int_{\Omega} \vt_t \Delta \varphi_t
     \le C \|\vt_t\|^2
       + \frac{1}{6} \|\Delta \varphi_t\|^2.
\]
Finally, recalling \eqref{hp:F3},
\[
  \int_{\Omega} F''(\varphi) \varphi_t \, \Delta \varphi_t
   \le C (1 + \|\varphi\|^{p_F}_{L^{\infty}(\Omega)}) \|\varphi_t\| \, \|\Delta \varphi_t\|
   \le \frac{1}{6} \|\Delta \varphi_t\|^2
     + C \|\varphi_t\|^2.
\]
Summarizing,
\[
  \ddt \|\varphi_t\|^2 + \|\Delta \varphi_t\|^2
    \le C (1 + \|\varphi_t\|^2 + \|\vu_t\|^2 + \|\vt_t\|^2 ).
\]
Hence, \eqref{stimavarphit} follows from Gronwall's lemma using the information
\eqref{dal:main3} (and estimating the boundary layer near~$0$ as before).

We now claim that this is enough to get the desired estimate (x):
indeed, from \eqref{CH1} we immediately obtain $\Delta \mu \in L^{\infty}(\tau, +\infty; H)$
(with a quantitative bound on that norm) because $\vu$ and $\nabla\fhi$ are bounded uniformly.
By elliptic regularity, this implies $\mu \in L^{\infty}(\tau, \infty; H^2(\Omega))$.
Then, we also interpret \eqref{CH2} as an elliptic problem, namely
we have $\Delta \varphi = g $, where $g \in L^{\infty}(\tau, \infty; H^2(\Omega))$
thanks to (ix), \eqref{hp:F4} and \eqref{lemma4.3}. Hence, we  deduce
that $\varphi \in L^{\infty}(\tau, + \infty; H^4(\Omega))$ and, more precisely,
the quantitative bound
$ \| \varphi \|_{L^{\infty}(\tau, + \infty; H^4(\Omega))} \le Q(\| z_0 \|_{\calV^r},\tau^{-1})$.
This concludes the proof of Lemma~\ref{ascpt-le}.
\end{proof}


\subsection{$\omega$-limits and dissipativity}
\label{sub:omega}

On account of the previous estimates, we can now show
that the dynamical process $S(t)$ associated to ``stable solutions'' is
{\sl asymptotically compact}. Namely, we have the
\bete\label{ascpt-te}
 Let $\{z\zzn\}$ be a bounded sequence in $\calV^r$ and let $z_n(t)=S(t) z\zzn$
 be the unique stable solution emanating from $z\zzn$. Then, for any sequence $t_n\nearrow +\infty$,
 there exist an element $\zeta\in\calV^r$ and a (nonrelabelled) subsequence of
 $t_n$ such that
 $$
    S(t_n) z\zzn \to \zeta \quext{in }\calV^r.
 $$
\ente
\begin{proof}
Thanks to Lemma~\ref{ascpt-le}, the sequence $\{ S(t_n)z\zzn \}$ is bounded in $\calW$.
Hence, recalling \eqref{norW} it is clear that, for some $\zeta=(\vu,\fhi,\vt)\in \calW$,
there holds (all the following relations are intended to hold up to the extraction
of nonrelabelled subsequences)
\begin{equation}\label{weakW0}
  S(t_n)z_{0,n} \to \zeta \quext{weakly in }H^2(\Omega)\times H^4(\Omega) \times H^2(\Omega),
\end{equation}
whence, by Rellich's theorem,
\[
  S(t_n)z_{0,n} \to \zeta \quext{strongly in }H^{1+r}(\Omega)\times H^3(\Omega) \times V,
   ~~\text{and uniformly in $\Omega$}.
\]
It is then easy to check that we also have $K(\vt_n)\to K(\vt)$ in $V$.
Moreover, $\{1/\vt_n\}$ is bounded in $W^{1,1}(\Omega)\cap L^4(\Omega)$. Hence,
using also pointwise convergence, we obtain that $1/\vt_n\to1/\vt$ strongly
in $L^1(\Omega)$ (actually, something more is true), which concludes the proof.
Note that, in fact, our argument shows that
the metric space embedding $\calW\subset \calV^r$ is compact.
\end{proof}
\beos\label{oss:W}
 In the sequel, with some abuse of notation, when a sequence $\{\zeta_n\}=\{(\vu_n,\fhi_n,\vt_n)\}\subset \calW$ satisfies
 \begin{equation}\label{weakW}
   \zeta_n \to \zeta \quext{weakly in }H^2(\Omega)\times H^4(\Omega) \times H^2(\Omega),
    \qquad 1/\vt_n \to 1/\vt \quext{weakly in }L^4(\Omega),
 \end{equation}
 for some $\zeta=(\bu,\fhi,\vt)\in H^2(\Omega)\times H^4(\Omega) \times H^2(\Omega)$ with $1/\vt \in L^4(\Omega)$,
 we will speak of ``weak convergence in $\calW$''. This is in fact the type of information
 we obtain from asymptotic compactness. As before, by Rellich's theorem, \eqref{weakW}
 implies strong convergence of the components in weaker norms.
\eddos
\noindent
In order to show existence of the global attractor in the case when the
spatial mean of the velocity is $\bzero$, we will combine the above property
with the {\sl point dissipativity} of the dynamical process generated by stable solutions.
Namely, we will prove that there exists a bounded set $\calB\subset \calV^r$ such that
for any $z_0\in \calV^r$ there exists $T=T(z_0)$ such that $S(t)z_0\in \calB$ for
any $t\ge T$. In other words, $\calB$ is a {\sl pointwise absorbing}\/ set.
For many evolutionary system, this kind of property (and often a stronger one,
i.e., the existence of a {\sl uniformly absorbing}\/ set) can be proved directly
by showing that suitable norms of the solutions satisfy a dissipative differential
inequality. Here, however, due to the presence of the quadratic source terms
in \eqref{calore} and to the physical constraints corresponding to conservation
of mass, momentum, and total energy, it seems difficult to derive such a type
of inequality. For this reason we will work in an alternative way, proving first
of all that any trajectory has a nonempty $\omega$-limit set, which
is contained in a {\sl proper subclass}\/ $\calS_0$
of the family of solutions of the stationary problem associated to
\eqref{incom}-\eqref{calore}. Pointwise dissipativity will follow from a precise characterization of $\mathcal{S}_0$.
%
In order to start with this program, we preliminarily observe that the
stationary problem associated to our system has the form
\begin{align}\label{incom-s}
  & \dive \vu = 0, \\
 \label{ns-s}
  & \vu \cdot \nabla \vu + \nabla p
    = \Delta \vu - \dive ( \nabla \fhi \otimes \nabla \fhi ),\\
 \label{CH1-s}
  & \vu \cdot \nabla \fhi = \Delta \mu, \\
 \label{CH2-s}
  & \mu = - \Delta \fhi + F'(\fhi) - \vt, \\
 \label{calore-s}
  & \vu \cdot \nabla \vt + \vt \Delta \mu
    - \dive(\kappa(\vt)\nabla \vt) = | \nabla \vu |^2 + | \nabla \mu |^2,
\end{align}
naturally complemented with periodic boundary conditions. Letting $\calS$ be
the set of solutions of \eqref{incom-s}-\eqref{calore-s}, we will now
show that any element of the $\omega$-limit set of a stable
solution of the evolution system not only belongs to $\calS$,
but it also satisfies additional structure properties in such a way that it solves
in fact a much simpler system. This is the object of the following
\bele\label{lem:omega}
 Let $z_0=(\vu_0,\fhi_0,\vt_0)\in \calV^r$ and let
 \begin{equation}\label{vincoli}
    \bm :=(\vu_0)\OO=\bzero, \qquad
    m :=(\fhi_0)\OO, \qquad
    M := \frac12 \| \vu_0 \|^2 + \frac12 \| \nabla\fhi_0 \|^2 + \io (F(\fhi_0)+\vt_0)
 \end{equation}
 be the associated conserved quantities. Let $z(t):=S(t)z_0$ and let $t_n\nearrow\infty$.
 Then, there exist a nonrelabelled subsequence of $t_n$ and
 $z_\infty=(\bzero,\fhi_\infty,\vt_\infty)\in \calV^r$
 such that
 \begin{equation}\label{o:lim}
   z(t_n)=S(t_n)z_0 \to z_\infty \quext{in }\calV^r
 \end{equation}
 (and, in fact, ``weakly'' in $\calW$).
 Moreover, there exists $\mu_\infty\in V$ such that
 \begin{equation}\label{o:lim2}
   \mu(t_n) \to \mu_\infty \quext{strongly in }V.
 \end{equation}
 In addition to that, if $z_\infty=(\vu_\infty,\fhi_\infty,\vt_\infty)$ (with the
 auxiliary variable $\mu_\infty$) is\/ {\rm any}
 limit point of a sequence $S(t_n)z_0$ in the sense specified above, then we have that
 $\vu_\infty$, $\mu_\infty$ and $\vt_\infty$ are constant
 functions with respect to the space variables with
 $\vu_\infty\equiv \bm\equiv \bzero$ and the stationary system reduces
 to the single equation
 \begin{align} 
  \label{CH2-red}
   & -\Delta\fhi_\infty + F'(\fhi_\infty) = \mu_\infty + \vt_\infty = \io F'(\fhi_\infty).
 \end{align}
 Finally, there exists a constant $C_\infty>0$ depending only on the conserved values $\bm=\bzero$,
 $m$ and $M$ (and hence independent of the specific choice of $z_0$) such that
 \begin{align}\label{regostaz}
   & |\mu_\infty| + | \vt_\infty | + \| \fhi_\infty \|_{H^4(\Omega)}
     \le C_\infty.
 %
 \end{align}
\enle
\begin{proof}
We start pointing out that all convergence relations
in the proof will be intended to hold up to the extraction of
nonrelabelled subsequences of $n$. That said, we set $z_n(t)=S(t)z(t_n)$, for
$t\in (0,1)$. Namely, we interpret $z(t_n)\in \calV^r$ as an ``initial'' datum and
consider the corresponding ``stable'' solution $z_n$ over the time interval $(0,1)$. Then, by
Lemma~\ref{ascpt-te} (asymptotic compactness), we have
\begin{equation}\label{om:11}
  z(t_n)\to z_\infty=(\vu_\infty,\fhi_\infty,\vt_\infty) \quext{in }\calV^r
   ~~\text{(and weakly in $\calW$)}
\end{equation}
for some $z_\infty\in \calV^r$. As a consequence, it is immediate to check that
\begin{equation}\label{vincolin}
   (\vu_\infty)\OO=\bm=\bzero, \quad
   (\fhi_\infty)\OO=m, \quad
    \frac12 \| \vu_\infty \|^2 + \frac12 \| \nabla\fhi_\infty \|^2 + \io (F(\fhi_\infty)+\vt_\infty) = M.
\end{equation}
Moreover, looking at the proof of~(x) in Lemma~\ref{ascpt-le}, we can easily
realize that $\mu(t_n)$ is bounded in $H^2(\Omega)$. Hence, \eqref{o:lim2}
holds.

Let us now look at the behavior of $z_n$ over the time interval $(0,1)$. Since
$z(t_n)$ is a convergent (sub)sequence in $\calV^r$, by weak sequential stability
of solutions (cf.~\cite{ERS2} for details),
it readily follows that, correspondingly, $z_n$ converges
in a proper way to a limit $\barz=(\barvu,\barfhi,\bar{\vartheta})$
and $\mu_n=\mu(t_n)$ converges to $\barmu$, where $\barz,\barmu$
are defined over $(0,1)\times \Omega$ and solve system \eqref{incom}-\eqref{calore}
with the periodic boundary conditions and the initial conditions
\begin{equation}\label{om:12}
  \barvu|_{t=0} = \vu_\infty, \qquad
   \barfhi|_{t=0} = \fhi_\infty, \qquad
   \barvt|_{t=0} = \vt_\infty.
\end{equation}
Let us now prove that $\barz$ is independent of time.
First of all, in view of the dissipation integrals \eqref{diss-int}, it is clear that
\begin{equation}\label{om:14}
   \nabla \vu_n, \nabla \vt_n, \nabla \mu_n \to 0 \quext{strongly in }L^2(0,1;H).
\end{equation}
Then, let us take $\xi\in H^2(\Omega)\subset L^\infty(\Omega)$
with $\| \xi \|_{H^2(\Omega)}\le 1$ and test \eqref{calore},
written for $z_n$ on the time span $(0,1)$, by $\xi$. Then,
performing standard manipulations and subsequently taking the supremum
with respect to such test function $\xi$, we infer
\begin{align*}
  \| \dt \vt_n \|_{H^2(\Omega)'}
   & \le \| \vu_n \| \, \| \nabla \vt_n \|
    + \| \nabla \mu_n \| \, \| \nabla \vt_n \|
    + \| \nabla\mu_n \| \, \| \vt_n \|_{L^4(\Omega)}\\
  & \qquad \qquad
    + \| \kappa(\vt_n) \|_{L^4(\Omega)} \, \| \nabla \vt_n \|
    + \| \nabla \mu_n \|^2
    + \| \nabla \vt_n \|^2\\
  & \le C\big( \| \nabla \mu_n \|^2
    + \| \nabla \vt_n \|^2
    + \| \nabla \vt_n \|
    + \| \nabla \mu_n \| \big),
\end{align*}
where $C$ may also depend on the uniformly estimated norms of the other quantities
(cf.~\eqref{dal:main2}, \eqref{dal:main3}).
Squaring and integrating over $(0,1)$, we deduce
\[
   \| \dt \vt_n \|_{L^2(0,1;H^2(\Omega)')}^2
 \le C \big( 1 + \| \nabla \vt_n \|_{L^\infty(0,1;H)}^2
         + \| \nabla \mu_n \|_{L^\infty(0,1;H)}^2 \big)
  \big( \| \nabla \vt_n \|_{L^2(0,1;H)}^2
         + \| \nabla \mu_n \|_{L^2(0,1;H)}^2 \big),
\]
whence, by \eqref{dal:main2} and \eqref{om:14},
$\dt \vt_n$ goes to $0$ strongly in $L^2(0,1;H^2(\Omega)')$
and consequently $\barvt\equiv \vt_\infty$ over $(0,1)$.
Moreover, using \eqref{om:14} again,
we see that $\vt_\infty$ does not depend on space
variables.

Next, we consider the behavior of the other variables, which is simpler to describe.
First of all, notice that, by \eqref{om:14}, the assumption $\bm=\bzero$, and the
Friedrichs inequality, we deduce that $\vu_n\to \bzero$
strongly in $L^2(0,1;V)$. More precisely, since $\vu_n$ converges
uniformly as a consequence of \eqref{dal:main3}, (viii) and the
Aubin-Lions lemma, it turns out that $\barvu\equiv \vu_\infty = \vm = \bzero$.
Moreover, we easily deduce from \eqref{CH1} that
\[
  \| \dt \fhi_n \|_{L^2(0,1;V')}
   \le C \big( \| \nabla \mu_n \|_{L^2(0,1;H)}
     + \| \vu_n \|_{L^2(0,1;V)} \, \| \fhi_n \|_{L^\infty(0,1;V)} \big),
\]
whence, using \eqref{dal:main2} and \eqref{om:14}, we deduce that $\dt \fhi_n \to 0$
strongly in $L^2(0,1;V')$. Consequently, $\barfhi$ is also independent of time
and, more precisely, $\barfhi\equiv \fhi_\infty$ over $(0,1)$.
Finally, concerning the auxiliary variable $\mu_n$,
from \eqref{om:14} we deduce that $\barmu$ is independent of
space variables. Note that, a priori, it is not clear whether $\barmu(t)$
is also independent of $t\in(0,1)$. On the other hand, taking $n\nearrow \infty$
in \eqref{CH2} and collecting the previous information, we deduce
that, a.e.~in~$(0,1)$ (in fact, everywhere in $[0,1]$ since one
can easily realize that also $\mu_n$ converges uniformly), there holds
\[
  \barmu = - \Delta \barfhi + F'(\barfhi) - \barvt
   \equiv - \Delta \fhi_\infty + F'(\fhi_\infty) - \vt_\infty.
\]
Then, in view of the fact that the \rhs\ is independent of time, so
is also the \lhs. Hence $\barmu\equiv\mu_\infty$ and
the above relation reduces to \eqref{CH2-red} (in particular
the second equality follows by integration over $\Omega$).

It remains to prove \eqref{regostaz}.
First of all, we observe
that, thanks to \eqref{vincolin} and \eqref{czero}, we have
\[
  | \vt_\infty | + \| \fhi_\infty \|_V \le c_{M,m}
\]
(again, we control the full $V$-norm of $\fhi_\infty$
thanks to the conservation of mass). By Sobolev's embeddings, we
then also deduce that $\| F'(\fhi_\infty) \|_{L^1(\Omega)} \le c_{M,m}$.
Hence, from \eqref{CH2-red},
\[
  | \mu_\infty |
  \le | \mu_\infty + \vt_\infty | + | \vt_\infty |
  \le \| F'(\fhi_\infty) \|_{L^1(\Omega)} + | \vt_\infty |
  \le c_{M,m}.
\]
Next, applying elliptic regularity in \eqref{CH2-red}
(or, equivalently, testing by $(-\Delta)^3\fhi_\infty$)
and exploiting \eqref{hp:F4}, we easily get the $H^4$-control
in \eqref{regostaz}, which concludes the proof.
\end{proof}
\noindent%
To conclude this part, we would like to briefly describe
the case when the initial velocity
$\vu_0$ has a nonzero spatial mean $\bm$. In this situation, thanks
to the periodic boundary conditions, we can actually
perform a change of variables. Namely, we denote
\[
  \tilde{\zeta}(t,x) := \zeta(t, x + t \bm ), \quext{for }
    \zeta = \vu,\fhi,\mu,\vt,p,
\]
and we can easily check that
\[
  \nabla \tilde{\zeta} (t,x) = \nabla \zeta(t, x + t \bm), \quad \qquad
    \tilde{\zeta}_t(t,x) = \zeta_t (t, x + t \bm) + \bm \cdot \nabla \zeta (t, x + t \bm).
\]
Then, if $(\vu,\fhi,\mu,\vt)$ is a stable solution,
the Cahn-Hilliard system \eqref{CH1}-\eqref{CH2} is transformed
into
\begin{align}
  & \tilde{\varphi}_t + (\tilde{\vu} -  \bm) \cdot \nabla \tilde{\varphi} = \Delta \tilde{\mu}, \label{CHtilde1}\\
  & \tilde{\mu} = -\Delta \tilde{\varphi} + F'(\tilde{\varphi}) - \tilde{\vt}. \label{CHtilde2}
\end{align}
Correspondingly, the counterparts of \eqref{ns} and \eqref{calore} are respectively
\begin{align}\label{nstilde}
  & \tilde{\vu}_t + (\tilde{\vu} - \bm) \cdot \nabla \tilde{\vu} + \nabla \tilde{p}
    = \Delta \tilde{\vu} - \dive (\nabla \tilde{\varphi} \otimes \nabla \tilde{\varphi}),\\
 \label{caloretilde}
  & \tilde{\vt}_t + (\tilde{\vu} - \bm) \cdot \nabla \tilde{\vt} + \tilde{\vt} \Delta \tilde{\mu}
  - \dive (\kappa(\tilde{\vt}) \nabla \tilde{\vt}) = |\nabla \tilde{\vu}|^2 + |\nabla \tilde{\mu}|^2
     = |\nabla (\tilde{\vu} - \bm )|^2 + |\nabla \tilde{\mu}|^2.
\end{align}
Next, we observe that \eqref{nstilde} can be also rewritten as
\[
   (\tilde{\vu} - \bm)_t + (\tilde{\vu} - \bm) \cdot \nabla  (\tilde{\vu} - \bm)  + \nabla \tilde{p}
     = \Delta (\tilde{\vu} - \bm)
      - \dive (\nabla \tilde{\varphi} \otimes \nabla \tilde{\varphi})
\]
and the initial conditions may be restated as
\begin{equation}\label{om:60}
  \tilde{\varphi}(0)  = \varphi_0, \qquad
  \tilde{\vt}(0) = \vt_0, \qquad
  (\tilde{\vu} - \bm)(0) = \vu_0 - \bm,
\end{equation}
where, obviously,
\begin{equation}\label{om:61}
  (\tilde\vu - \bm)\OO(t) \equiv (\tilde\vu_0 - \bm)\OO = 0.
\end{equation}
Hence, the variables $\tilde{z}=(\tilde{\vu}-\bm, \tilde{\varphi}, \tilde{\vt})$, with the
auxiliary $\tilde\mu$, constitute a stable solution to \eqref{incom}-\eqref{calore}
with the initial conditions \eqref{om:60}. Hence Lemma~\ref{lem:omega} applies to
$\tilde{z}$, which admits a nonempty $\omega$-limit all of whose elements
belong to $\calS_0$. In particular, if $\{t_n\}$ is a diverging sequence of times, we
deduce for instance that (a subsequence of) $\tilde{\vt}(t_n)$ converges
weakly in $H^2(\Omega)$ (whence uniformly in $\Omega$) to a positive constant
$\vt_\infty$. On the other hand,
\begin{equation}\label{om:62}
  \vt(t_n,x) = \tilde{\vt}(t_n,x-t_n\bm)
   \to \vt_\infty,
\end{equation}
still uniformly in $\Omega$, and the same applies to $\mu(t_n)$.
%
%
%
%
On the other hand, as far as the limit of $\fhi$ is concerned,
the situation is more intricated. Indeed, by asymptotic compactness
we can always construct a (nonrelabelled)
subsequence of $n\nearrow\infty$ such that $\fhi(t_n)$ tends to a suitable
limit $\fhi_\infty$ and simultaneously $\tilde{\fhi}(t_n)$ tends to some
limit $\tilde{\fhi}_\infty$. Moreover, convergence holds in both cases
in $H^3(\Omega)$, hence uniformly in $\Omega$. On the other hand,
if we try to characterize
$\fhi_\infty$ as a stationary state, we have to consider that
\begin{equation}\label{om:63}
  \fhi(t_n,x) = \tilde{\fhi}(t_n,x-t_n\bm),
\end{equation}
whence, for any $\epsilon>0$, there exists $\overline{n}(\epsilon)$
such that
\begin{align}\no
  | \fhi_\infty(x) - \tilde{\fhi}_\infty(x-t_n \bm) |
   & \le \| \fhi_\infty - \fhi(t_n,\cdot) \|_\infty
    + \| \tilde{\fhi}(t_n,\cdot-t_n\bm) - \tilde{\fhi}_\infty(\cdot-t_n \bm) \|_\infty\\
 \label{om:64}
   & \le 2\epsilon \quext{for all }\,n\ge \overline{n}(\epsilon)
\end{align}
and for every $x\in\Omega$. Now, $\tilde{\fhi}_\infty$ is a stationary solution
(more precisely, it belongs to the subclass $\calS_0$)
in view of Lemma~\ref{lem:omega}. Then, in order to deduce a useful information from the
above relation, we may assume that the subsequence of $n$ is chosen in such a way
that, also, $t_n \bm \to x_0$ for some $x_0\in \Omega$ (of course we are
using here the fact that the flat torus is a compact manifold). Hence, we eventually deduce
that
\begin{equation}\label{om:65}
  \fhi_\infty(x) = \tilde{\fhi}_\infty(x-x_0).
\end{equation}
Namely, we have proved that, if $\{t_n\}$ is a diverging sequence of times
and $z$ is a stable solution with $\bm\neq\bzero$, then a suitable subsequence
of $\fhi(t_n)$ tends to a limit $\fhi_\infty$ such that
\begin{equation}\label{om:66}
  - \Delta \fhi_\infty(\cdot+x_0) + F'(\fhi_\infty(\cdot+x_0)) = \mu_\infty + \vt_\infty
\end{equation}
for some $x_0\in \Omega$, where of course $x_0$ may depend on the chosen subsequence of $t_n$,
and where the constants $\mu_\infty$, $\vt_\infty$ are characterized as in the lemma.


\subsection{End of proof of Theorem~\ref{teo:att}}
\label{sub:conclu}

Along this section, in view of Remark~\ref{rem:quad}, we will often
intend solutions as a quadruples instead of triples.
So, with a small abuse of language, we will sometimes write
$z=(\vu,\fhi,\mu,\vt)$ in place of $z=(\vu,\fhi,\vt)$.
We also recall that here we just consider the case $\bm=\bzero$.
As in Lemma~\ref{lem:omega}, we will
note as $\calS_0=\calS_0(\bzero,m,M)$ the set of all the
quadruples $(\vu_\infty,\fhi_\infty,\mu_\infty,\vt_\infty)$, with
$\vu_\infty=\bzero\in \RR^2$, $\mu_\infty,\vt_\infty\in \RR$, $\fhi_\infty\in H^4(\Omega)$,
satisfying equation \eqref{CH2-red} as well as the bound
\eqref{regostaz} and the constraints
\begin{equation}\label{vincolin2}
   \vu_\infty \equiv\bzero, \quad
   (\fhi_\infty)\OO=m, \quad
    \frac12 | \vu_\infty |^2 + \frac12 \| \nabla\fhi_\infty \|^2 + \io (F(\fhi_\infty)+\vt_\infty) = M.
\end{equation}
We also recall that, if $z_\infty$ is any element of the $\omega$-limit of a
stable solution, then $z_\infty\in \calS_0$ automatically satisfies
the estimate \eqref{regostaz} with $C_\infty$ as in Lemma~\ref{lem:omega}.
However, the values $\mu_\infty$ and $\vt_\infty$ are
not univocally determined. Namely, even if $\bm=\bzero$, $m$ and $M$ are assigned,
different elements of the $\omega$-limit of a solution emanating from
a datum $z_0\in\calV^r_{\bzero,m,M}$ may solve \eqref{CH2-red} for different values
of $\mu_\infty$ and $\vt_\infty$. Physically speaking, the limit value
$\vt_\infty$ of the $\vt$-component of a trajectory tells us how much chemical
and kinetic energy is converted into heat. This does not just depend on the
initial value of the energy and of the mass, but also, for instance,
on how far the initial value $\fhi_0$ is from the chemical equilibrium.
For the same reason, we expect that, in general, we cannot obtain
a {\sl lower}\/ bound for $\vt_\infty$ holding uniformly
for all initial data lying in $\calS_0=\calS_0(\bzero,m,M)$.
In other words, it may happen that, for fixed values of the mass $m$ and of the energy
$M$, there exist initial data in $\calV^r_{\bzero,m,M}$ for which the ``limit temperature''
$\vt_\infty$ may be arbitrarily close to $0$. This may in principle happen
if the initial temperature is also close to $0$ and the initial chemical configuration
is so favorable that no energy (or almost no energy) is converted into heat.

So, it is to avoid this situation that we need to restrict ourselves to those
configurations for which the initial entropy is greater than some assigned value $-R$,
namely, for initial data lying in $\calV^{r,R}:=\{z\in \calV^r:~( - \log \vt)\OO \le R\}$.
Then we can first prove that $\calV^{r,R}$ is invariant for $S(t)$:
\bele\label{lem:omega2}
 Let $z_0\in \calV^{r,R}$. Then,
 $z(t)=S(t)z_0 \in \calV^{r,R}$ for all $t\ge 0$. Moreover,
 for any element $(\vu_\infty,\fhi_\infty,\mu_{\infty},\vt_\infty)$  of the
 $\omega$-limit of $z(t)$ we have
 \begin{equation}\label{regostaz2}
    \vt_\infty \ge e^{-R}.
 \end{equation}
\enle
\begin{proof}
The invariance property follows from the entropy inequality \eqref{controllo_log}.
Then, by Jensen's inequality we also have
\[
  -\log( \vt(t)\OO )
  \le (- \log \vt(t) )\OO \le R \quad \perogni t\ge 0,
\]
whence, letting $t_n\nearrow \infty$ in such a way that $z(t_n)$ tends to 
an element of the $\omega$-limit, the second assert follows.
\end{proof}
\noindent%
We can now complete the proof of Theorem~\ref{teo:att}. Actually, thanks to
the above invariance property, if $z_0\in \calV^{r,R}_{\bzero,m,M}$,
any element $z_\infty$ of its $\omega$-limit set not only lies in $\calS_0$, but
it also satisfies $z_\infty\in\calV^{r,R}_{\bzero,m,M}$. We can then set
$\calS_0^R:=\calS_0\cap \calV^{r,R}_{\bzero,m,M}$. As a consequence, the limit
temperature satisfies \eqref{regostaz2}, namely, it is bounded from below by
a quantity that only depends on the phase space and is
actually independent of the specific choice of the initial datum.

Thanks to this fact, we may prove existence of the global attractor.
As remarked before, we will combine the {\sl asymptotic compactness}\/
of the semiflow $S(t)$ proved in
Theorem~\ref{ascpt-te} with the following {\sl point dissipativity}\/ property:
\bete\label{pd-te}
 Let $m,M,R\in \RR$ be assigned. Then there exists a bounded subset
 $\calB=\calB(\bzero,m,M)\subset \calV^{r,R}_{\bzero,m,M}$ such that, for any
 $z_0=(\vu_0,\fhi_0,\vt_0)\in \calV^{r,R}_{\bzero,m,M}$, there exists
 $T_0=T_0(z_0)$ such that $S(t)z_0 \in \calB$ for all $t\ge T_0$.
\ente
\begin{proof}
As usual, we proceed by contradiction. Note that,
by Lemmas~\ref{lem:omega}, \ref{lem:omega2},
we have
\begin{equation}\label{regostar}
  | \mu_\infty| + | \vt_\infty | + \| \fhi_\infty \|_{H^4(\Omega)}
    \le C_\infty, \qquad \vt_\infty \ge e^{-R}
\end{equation}
for any $z_\infty=(\vu_\infty,\fhi_\infty,\mu_\infty,\vt_\infty)$ belonging
to $\calS_0^R$. We introduce $\calB$ as a sort of neighbourhood
of $\calS_0^R$ in $\calV^r_{\bzero,m,M}$.
Namely, we define $\calB$ as the set of those functions $(\vu,\fhi,\vt)\in \calV^r$
such that $\vu\OO=\bzero$, $\fhi\OO=m$, $\calE(\vu,\fhi,\vt)=M$ (recall that $\mu$ plays the role of an auxiliary variable),
and there exists $(\vu_{\infty},\fhi_\infty,\mu_\infty,\vt_\infty)\in \calS_0^R$
such that \eqref{vincolin2} holds together with
\begin{equation}\label{om:*3}
   \| \vu \|_{H^{1+r}(\Omega)}
    + \| \fhi - \fhi_\infty \|_{H^{3}(\Omega)} \le 1,
  \qquad \| \vt - \vt_\infty \|_{H^{3/2}(\Omega)}
  \le c_*,
\end{equation}
where $c_*$ is chosen in the following way: denoting as $c_\Omega$ the embedding
constant of $H^{3/2}(\Omega)$ into $C(\Omega)$ we ask that
$c_* c_\Omega \le \frac{e^{-R}}4$ in such a way that if $\vt$
is the last component of an element $z\in \calB$, then
for every $x\in \Omega$ there holds
\begin{equation}\label{dist:00}
   \vt(x) \ge \vt_\infty - | \vt(x) - \vt_\infty |
    \ge e^{-R} - c_\Omega \| \vt - \vt_\infty \|_{H^{3/2}(\Omega)}
    \ge \frac34 e^{-R}.
\end{equation}
Hence, $\vt$ is separated from $0$ in the uniform norm.
Let us now assume, by contradiction, that there exist $z_0\in \calV^{r,R}_{\bzero,m,M}$
and a sequence $\{t_n\}$ with $t_n\nearrow\infty$ such that $S(t_n)z_0\not\in \calB$
for any $n\in\NN$. By asymptotic compactness there exists a (nonrelabelled)
subsequence such that $S(t_n)z_0\to z_\infty$ in $\calV^r$
(and ``weakly'' in $\calW$) with $z_\infty\in\calS_0^R$
(and the last component $\vt_\infty$ satisfies \eqref{regostaz2}).
Note that we actually used the fact $\vt(t_n)\to \vt_\infty$
in $H^{3/2}(\Omega)$, which is true because we have
asymptotic boundedness of $\vt(t_n)$ in $H^2(\Omega)$.
This gives a contradiction.
\end{proof}
\noindent%
To prove existence of the global attractor, we need to exhibit however the existence
of a {\sl uniformly absorbing}\/ set, i.e., of a set that eventually contains
bundles of trajectories starting from any bounded set in the phase space. This
is a stronger property with respect of that shown above (indeed, $\calB$ only
absorbs single trajectories). This fact is proved in the following lemma
that adapts the ideas of a classical argument (see, e.g., \cite{B97,Hale}):
\bele\label{lem:omega3}
 Let the assumptions of the previous lemma hold and let $R>0$.
 Let $\calB_0$ be defined as the set of those triples
 $(\vu,\fhi,\vt)\in \calV^{r,R}_{\bzero,m,M}$ such that
 there exists $(\barvu,\barfhi,\barvt)\in \calB$ with
 \begin{equation}\label{om:*4}
    \| \vu - \barvu \|_{H^{1+r}(\Omega)}
     + \| \fhi - \barfhi \|_{H^{3}(\Omega)}
     + \| \vt - \barvt \|_{H^{3/2}(\Omega)} < c_*.
 \end{equation}
 Then, let also
 \begin{equation}\label{om:*5}
    \calB_1:= \bigcup_{t\ge 0} S(t)\calB_0.
 \end{equation}
 Then, $\calB_1$ is a bounded in $\calV^{r}$ and is a uniformly absorbing set.
\enle
\begin{proof}
First of all, note that, by construction, if $(\vu,\fhi,\vt)\in \calB_0$,
then $\vt(x)> e^{-R}/2$ for every $x\in\Omega$. Hence, in particular,
$\calB_0$ is a bounded subset of $\calV^r$ and, consequently, $\calB_1$ is also bounded in $\calV^r$ by virtue
of the uniform estimates (recall, e.g., \eqref{ACtotal}).
Moreover, $\calB_0$ is an open set in $H^{1+r}(\Omega)\times H^3(\Omega) \times  H^{3/2}(\Omega)$
because it is the union over $(\barvu,\barfhi,\barvt)$ ranging in $\calB$ of triples satisfying
the strict inequality \eqref{om:*4}.

Let now $B$ be a given bounded set of $\calV^{r,R}_{\bzero,m,M}$.
We then claim that there exists $T_B$ such that, for any $t\ge T_B$,
$S(t) B\subset \calB_1$. To prove this fact, let us argue once more by contradiction.
Namely, let us assume there exist a sequence $\{z_n\}\subset B$ and a sequence
$t_n\nearrow\infty$ such that $S(t_n)z_n\not\in \calB_1$ for any $n\in\NN$.
Then, we observe that, due to uniform smoothing of
trajectories, $S(1)B$ is bounded in $\calW$ and consequently relatively
compact in $\calV^r$. Hence, we can assume that (up to a subsequence)
$S(1)z_n$ tends to some limit $z^1$, say,
in $\calV^r$ (of course a stronger convergence
holds). We claim that $S(t)z^1\not\in \calB_0$ for any $t\ge 1$.
Indeed, let us fix $t\ge 1$. Then, at least for $n$ large enough,
we have that $S(t_n)z_n=S(t_n-t-1) S(t) S(1) z_n$ and $S(t)S(1)z_n$
cannot lie in $\calB_0$, otherwise, by \eqref{om:*5}, it would be $S(t_n)z_n\in \calB_1$,
a contradiction. Hence, $S(t)S(1)z_n$ does not lie in $\calB_0$,
at least for $n$ large enough depending on the chosen $t$.
This means that, for any $n$ large enough, $S(t)S(1)z_n$
lies in the complement of $\calB_0$, which is a closed set
in $H^{1+r}(\Omega)\times H^3(\Omega) \times H^{3/2}(\Omega)$.
In view of the facts that $S(1)z_n$ tends to $z^1$
in $H^{1+r}(\Omega)\times H^3(\Omega) \times H^{3/2}(\Omega)$
and of the continuity of the operator $S(t)$ from that space into
itself, we can take the limit $n\nearrow\infty$ to deduce that
$S(t)z^1$ lies in the complement of $\calB_0$ for all $t\ge 0$.
On the other hand, because $z^1 \in \calV^{r,R}_{\bzero,m,M}$,
the trajectory starting from $z^1$ has a nonempty $\omega$-limit
set all of whose elements lie in $\calS_0^R$,
hence in $\calB_0$, which gives a contradiction.
\end{proof}



\end{document}